\documentclass[final]{siamltex}
\usepackage{graphicx}
\usepackage{multirow,bigstrut}
\usepackage{color}
\usepackage{amsmath,amsfonts,amssymb}

\newcommand{\bx}{\mathbf{x}}

\newcommand{\md}{\mathrm{d}}

\title{Parallel energy-stable solver for a coupled
Allen--Cahn and Cahn--Hilliard system\thanks{Data: July 14, 2020. Corresponding author:
Chao Yang ({\tt chao\_yang@pku.edu.cn}).} }

\author{Jizu Huang\thanks{LSEC, Academy of Mathematics and Systems Science, Chinese Academy of Sciences, Beijing 100190, China
    and School of Mathematical Sciences, University of Chinese Academy of Sciences, Beijing 100049, China  ({\tt huangjz@lsec.cc.ac.cn}).}
  \and Chao Yang\thanks{School of Mathematical Sciences, Peking University, Beijing 100871, China ({\tt chao\_yang@pku.edu.cn}). }
   \and Ying Wei\thanks{Institute of Software, Chinese Academy of Sciences, Beijing 100190,  China ({\tt weiying14@iscas.ac.cn}).} }

\begin{document}

\maketitle
\begin{abstract}
{ In this paper, we study numerical methods for solving the coupled Allen--Cahn/Cahn--Hilliard system associated with a free energy functional of  logarithmic type. To tackle the challenge posed by the special free energy functional, we propose  a method to approximate the  discrete variational derivatives in polynomial forms, such that the corresponding finite difference scheme is  unconditionally energy stable and the energy dissipation law is maintained. To further improve the performance of the algorithm, a modified adaptive time stepping strategy is adopted such that the time step size  can be flexibly controlled  based on the dynamical evolution of the problem. To achieve high performance on parallel computers, we introduce a domain decomposition based, parallel Newton--Krylov--Schwarz method to solve the nonlinear algebraic system constructed from the discretization at each time step.  Numerical experiments  show that the proposed algorithm is second-order accurate in both space and time,  energy stable with large time steps, and  highly scalable to over ten thousands  processor cores on the Sunway TaihuLight supercomputer.
}
\end{abstract}



\section{Introduction}

To describe the evolution of coexistent phases in binary alloy systems 
that exhibit simultaneous phase separation and one or
more order-disorder transitions,  it is often of great interest to study the solution of an   Allen--Cahn/Cahn--Hilliard (AC/CH)  system, originally introduced 
by Cahn and Novick--Cohen \cite{Phys94.15}.
The AC/CH system can be obtained by taking the quasi-continuum limits for the free energy defined on the lattice,
\color{black} so that the dynamics of a binary alloy \color{black}  is  reduced to a gradient flow    as follows:
\begin{equation}
\label{eq:1}
\left\{
\begin{aligned}
\frac{\partial {u}}{\partial t} &= \nabla \cdot c(u,v)\nabla\frac{\delta {\cal E}}{\delta {u}},\\
\frac{\partial {v}}{\partial t} &=-\frac{c(u,v)}{\rho}\frac{\delta {\cal E}}{\delta {v}},
\end{aligned}
\right.
\end{equation}
where $u$, $v$ are functions on  $(\mathbf{x},t)\in\Omega\times[0,{\cal T}]$, and $\frac{\delta {\cal E}}{\delta {u}}$, $\frac{\delta {\cal E}}{\delta {v}}$ are the variational 
derivatives \color{black}(in the $L^2$ inner product) \color{black}  of the total free energy functional ${\cal E}$. 
The first equation in \eqref{eq:1} is the Cahn--Hilliard equation \cite{cahn-hilliard}, in which $u$ represents a conserved 
concentration field for the phase separation. The second equation in \eqref{eq:1} is the Allen--Cahn 
equation \cite{allen-cahn}, in which $v$ denotes a non-conserved order parameter for the anti-phase coarsening. 
In the AC/CH system, $c(u,v) = u(1-u)(1/4-v^2)$ is the mobility, which is degenerate at pure phases, and the density~$\rho$ is a positive constant. The total free energy functional can be formulated as
\begin{equation}
\begin{aligned}\label{eq:2}
{\cal E}(u,v)= \int_{\Omega}e(u,v)\md\bx,
\end{aligned}
\end{equation}
where the local free energy is
$$e(u,v)=\frac{\alpha}{2}u(1-u)-\frac{\beta}{2}v^2 
          + \theta(\Phi(u+v)+\Phi(u-v)) + \frac{\gamma}{2}(\lvert \nabla u \rvert^2 + \lvert \nabla v \rvert^2)$$
   and $\Phi(z)=z\textnormal{ln}z + (1-z)\textnormal{ln}(1-z)$. Here positive constants $\gamma$, $\theta$, $\alpha$, 
$\beta$ are  corresponding to the coefficients of the gradient energy, the entropy, 
the nearest neighbors pairwise energetic interactions,
and next-nearest neighbors pairwise energetic interactions, respectively. 
\color{black} We consider periodic 
boundary conditions or the following homogeneous Neumann boundary conditions \cite{Phys94.15}
$$ \mathbf{n}\cdot \nabla u |_{\partial \Omega} =  \mathbf{n}\cdot\nabla v|_{\partial \Omega} = \mathbf{n} \cdot c \nabla \frac{\delta {\cal E}}{\delta u} \Big|_{\partial \Omega} = 0, $$
where $\mathbf{n}$ is the outward normal of $\partial \Omega$.
\color{black}

To solve the AC/CH system (1), many early works  \cite{Barrett1, Barrett2, Millett1, Rokkam1, Xia1} made use of explicit time stepping schemes,
in which very small time steps were usually used 
due to the severe stability  restriction,  
thus making
the explicit methods impractical,
especially for long time simulations. To relax the restriction of time step size, 
an implicit finite element framework together with Jacobian-free Newton Krylov method were introduced to solve the coupled AC/CH system in \cite{Tonks1, Wang1},
in which the free energy defined in \eqref{eq:2} was replaced by simplified  ones, such as the double well free energy functional.
In \cite{yangcai12dd21}, a fully implicit method based on Newton--Krylov--Schwarz (NKS) algorithm was employed to solve the coupled AC/CH system
with the  free energy defined in \eqref{eq:2}.
Those implicit methods can use relatively larger time steps, but no energy stability analysis 
was provided. Therefore, the choice of time step size could easily violate the free-energy dissipation law of 
the AC/CH systems. It is worth pointing out that only one and two-dimensional numerical experiments were conducted 
in the aforementioned works, except for  \cite{Tonks1}. 
Therefore, it is of great interest to study how to efficiently solve the three-dimensional AC/CH system, with a  scheme
that  is energy stable and obeys the free-energy dissipation  law.

As is well known, the AC/CH system can be viewed as a gradient flow driven by a free energy. For gradient flows,
there exist several popular approaches, such as  convex splitting \cite{Baskaran, Elliott, Eyre, Shen1},  stabilization \cite{Jzhu, Shen2}, 
exponential time differencing \cite{LJu},  invariant energy quadratization \cite{XYang, JZhao}, scalar auxiliary variable \cite{Shen3}, and so on.
However, due to the existence of the logarithmic function in the free energy of the AC/CH system, it is quite difficult to extend the above approaches 
to the AC/CH system. Recently, a new approach for gradient flow problems, namely the discrete variational derivative (DVD) method, has been proposed and 
 successfully applied to Cahn--Hilliard, Allen--Cahn, and phase field crystal equations \cite{DVDM, PFCDVD}.
The most important advantage of DVD method is that the numerical scheme, by careful design, is able to keep several important properties 
of the original system, for instance, free-energy dissipation, free-energy conservation, mass conservation, etc.

However,
with the existence of the logarithmic term in the free energy functional, the discrete variational derivatives obtained by
a DVD method will introduce terms in rational forms, of which the denominators are usually very close to zero.  
Due to this difficulty, numerical calculation of the discrete variational derivatives is often unstable and inaccurate. 
As a result, the DVD method is not directly applicable to the AC/CH system.  
In this work, we propose a method to approximate the discrete variational derivatives of the AC/CH system 
so that the numerical instability of the DVD discretization is avoided.
With the approximation, an implicit scheme is constructed, which is unconditionally energy stable and therefore obeys the energy dissipative law. 
\color{black} Here, unconditionally energy stability means that the free energy is nonincreasing in time, regardless of the time step size \cite{BPVollmayr, SMWise}. \color{black}
To solve the large sparse
nonlinear algebraic system arising  at every implicit time step, we present a parallel, highly scalable,
NKS algorithm \cite{NKS}. 
Due to the multiple time scales exhibited by the AC/CH system, using fixed time step size is no longer practical. To that end, we
 propose an adaptive time stepping strategy modified from \cite{PFCDVD, zhang2013}, and verify the efficiency  of the proposed method by a series of experiments.

The remainder of this paper is organized as follows. 
In Sec. 2, an implicit unconditionally energy stable scheme for the AC/CH is constructed. 
In Sec. 3, we introduce the NKS algorithm together with the adaptive time stepping strategy to solve the 
nonlinear system. Experiment results on several two or three dimensional test cases are reported in Sec. 4 
and some concluding remarks are given in Sec. 5.

\section{Discretization for the Allen--Cahn/Cahn--Hilliard system}
First, we rewrite the AC/CH system as
\begin{equation}
\label{system-1}
\frac{\partial \mathbf{U}}{\partial t} = -{\cal A}\mathbf{G},
\end{equation}
 in  which 
$\mathbf{U}=(u,\,v)^T$,  ${\cal A}=\textnormal{diag}(-\nabla \cdot c(u,v)\nabla,{c(u,v)}/{\rho})$,
and $\mathbf{G}:=\frac{\delta {\cal E}}{\delta \mathbf{U}}$ represents the variational derivative of the free 
energy functional ${\cal E}$ with respect to $\mathbf{U}$. 
The local free energy $e(u,v)$ is decomposed into three parts, which are  
\begin{equation}
\label{neq:01}
\begin{aligned}
& e_1(\mathbf{U})= \frac{\alpha}{2}u(1-u)-\frac{\beta}{2}v^2,\\
& e_2 (\mathbf{U})= \theta(\Phi(u+v)+\Phi(u-v)),\\
& e_3 (\nabla \mathbf{U}^T)= \frac{\gamma}{2}(\lvert \nabla u \rvert^2 + \lvert \nabla v \rvert^2).
\end{aligned}
\end{equation}
In \eqref{neq:01}, $e_1+ e_2$ and $e_3$ represent the bulk energy and the interfacial energy, respectively.
The existence of the logarithmic term in $e_2$ is due to an ideal mixing of entropy for binary alloy \cite{Phys94.15}.
Here $u$ and $v$ should satisfy the following 
restrictions: $u\in(0,1)$, $v\in (-\frac12, \frac12)$, and $(u\pm v)\in (0,1)$. 
In the AC/CH system, $\mathbf{G}:=\frac{\delta {\cal E}}{\delta \mathbf{U}}=
\left(\frac{\delta {\cal E}}{\delta u},\,\frac{\delta {\cal E}}{\delta v}\right)^T$ is defined as 
\begin{equation}
\begin{aligned}
&~\frac{\delta {\cal E}}{\delta u} = \frac{\partial e}{\partial u} - \nabla \cdot \frac{\partial e}{\partial (\nabla u)}= -\alpha(u-1/2) + \theta \Phi'(u+v) +\theta \Phi'(u-v) -\gamma \Delta u,\\
&~\frac{\delta {\cal E}}{\delta v} = \frac{\partial e}{\partial v} -\nabla \cdot \frac{\partial e}{\partial (\nabla v)}= - \beta v + \theta \Phi'(u+v) -\theta \Phi'(u-v) -\gamma \Delta v.
\end{aligned}
\end{equation}
It is easy to check that ${\cal A}$ is a semi-positive operator for $u\in(0,1)$, $v\in (-\frac12, \frac12)$. 

By denoting $e_\delta:=e(\mathbf{U}+\delta \mathbf{U},\nabla \mathbf{U}^T + \delta \nabla \mathbf{U}^T)$, 
the following integral relationship between the variational derivative and the free energy holds
\begin{equation}
\label{condG}
\begin{aligned}
&{\cal E}(\mathbf{U}+\delta \mathbf{U},\nabla \mathbf{U}^T + \delta \nabla \mathbf{U}^T) - {\cal E}(\mathbf{U},\nabla \mathbf{U}^T) = \int_\Omega e_\delta - e(\mathbf{U},\nabla \mathbf{U}^T) \md\bx\\ 
&\approx \int_\Omega \left(\frac{\partial e_1}{\partial \mathbf{U}} \cdot \delta \mathbf{U} + \frac{\partial e_2}{\partial \mathbf{U}} \cdot \delta \mathbf{U} + 
 \left(\frac{\partial e_3}{\partial \nabla u}\cdot \delta \nabla u + \frac{\partial e_3}{\partial \nabla v}\cdot \delta \nabla v\right)\right)\md\bx \\
&=\int_\Omega \left(\frac{\partial e_1}{\partial \mathbf{U}} + \frac{\partial e_2}{\partial \mathbf{U}} 
- \gamma(\Delta u,\,\Delta v)^T\right)  \cdot \delta \mathbf{U} \md\bx + B \\
&= \int_\Omega\mathbf{G} \cdot \delta \mathbf{U}\md\bx+B.
\end{aligned}
\end{equation}
The first equality is obtained by using Taylor expansion, and the second equality comes from 
the integration-by-parts formula. Here $B$ represents the boundary terms from the integration-by-parts formula
and equals zero due to the given boundary conditions. Equation~\eqref{condG} shows the connection between the variational derivative 
and the free energy, and  plays an important role in the construction of the energy-stable numerical scheme for the AC/CH system.

The free energy of the AC/CH system satisfies the following equation  
\begin{equation}
\label{DIS1}
\begin{aligned}
  \frac{d}{dt}{\cal E}(\mathbf{U}, \nabla \mathbf{U}^T) &= \frac{d}{dt}\int_\Omega e (\mathbf{U}, \nabla \mathbf{U}^T)\md\bx \\
  &= \int_\Omega \left(\frac{\partial e}{\partial \mathbf{U}}\cdot \frac{\partial \mathbf{U}}{\partial t}+
 \frac{\partial e}{\partial (\nabla u)}\cdot \frac{\partial (\nabla u)}{\partial t}+
\frac{\partial e}{\partial (\nabla v)}\cdot \frac{\partial (\nabla v)}{\partial t}
  \right)\md\bx \\
  &= \int_\Omega\mathbf{G}\cdot\frac{\partial \mathbf{U}}{\partial t} \md\bx + B_1 = - \int_\Omega\mathbf{G}\cdot{\cal A} \mathbf{G}\md\bx + B_1,
\end{aligned}
\end{equation}
where $B_1$ is the boundary terms due to integration-by-parts, which vanishes with the given boundary conditions. 
Since ${\cal A}$ is a semi-positive operator,  it follows that
 $ \int_\Omega\mathbf{G}\cdot{\cal A} \mathbf{G}\md\bx \geq 0$. Thus we have
\begin{equation}
\frac{d}{dt}{\cal E} (\mathbf{U}, \nabla \mathbf{U}^T) \leq 0,
\end{equation}
which demonstrates the energy dissipation of the AC/CH system.

Without loss of generality, consider discretizing the AC/CH system on $\Omega=[0,L_x]\times[0,L_y]$, which is covered by a uniform mesh with mesh sizes 
$\Delta x = L_x/N_x$ and $\Delta y = L_y/N_y$.
The temporal interval $[0, {\cal T}]$ is split by a set of nonuniform points $\{t_n\}_0^{N_{\cal T}}$
with time steps $\Delta t_n=t_{n+1}-t_n$. Denote 
$\tilde{\mathbf{U}}^n_{i,j}\approx \mathbf{U}(x_i, y_j, t_n)$ as the approximate solution of AC/CH system
at the $n$-th time step, where
$(x_i, y_j) = \left((i-\frac{1}{2})\Delta x, (j-\frac{1}{2})\Delta y\right)$, $1\leq i \leq N_x, 1\leq j \leq N_y$. 
Throughout the paper, notations with tilde are the corresponding approximate solutions or functions at the discrete level.
 Let us introduce some useful notations as following 
$$[D_x^+\tilde\phi]^n_{i,j}=\frac{\tilde\phi^n_{i+1,j}-\tilde\phi^n_{i,j}}{\Delta x},\ \ [D_x^-\tilde\phi]^n_{i,j}=\frac{\tilde\phi^n_{i,j}-\tilde\phi^n_{i-1,j}}{\Delta x},
$$
 $$[D_x \tilde\phi]^n_{i,j}=\frac{\tilde\phi^n_{i+\frac{1}{2},j}-\tilde\phi^n_{i-\frac{1}{2},j}}{\Delta x}, \ \ 
 \left[(D_x^\pm\tilde\phi)^2\right]^n_{i,j}=\frac{\left([D_x^+\tilde\phi]^n_{i,j}\right)^2+\left([D_x^-\tilde\phi]^n_{i,j}\right)^2}{2}.$$
$[D_y^+\tilde\phi]^n_{i,j}, [D_y^-\tilde\phi]_{i,j}^n, 
 [D_y\tilde\phi]_{i,j}^n, \left[(D_y^\pm\tilde\phi)^2\right]^n_{i,j}$ are defined similarly. 
\color{black} In what follows, let us denote $\tilde \phi$ and $\tilde \varphi$ as the approximations of scalar value functions $\phi$ and $\varphi$, respectively. \color{black}
We denote the discrete gradient operator as $[\nabla\tilde\phi]_{i,j}^n=([D_x\tilde\phi]_{i,j}^n, [D_y\tilde\phi]_{i,j}^n)^T$.
The Laplacian operator $\Delta$ is discretized by
$[\Delta\tilde\phi]_{i,j}^n=[\nabla \cdot \nabla\tilde\phi]_{i,j}^n$.
The operator $\nabla \cdot c(u,v) \nabla$ in $\cal A$ is discretized by
$[(\nabla \cdot c(u,v) \nabla)\tilde\phi]_{i,j}^n =  [D_x \tilde c D_x \tilde\phi]_{i,j}^n
+ [D_y \tilde c D_y\tilde\phi]_{i,j}^n,$
where the approximated value of the mobility $c$ 
is calculated as
$$\tilde c^n_{i+\frac{1}{2},j} = \frac{c(\tilde u^n_{i+1,j}, \tilde v^n_{i+1,j}) + c(\tilde u^n_{i,j},\tilde v^n_{i,j})}{2}.$$
The discrete operator $\tilde{{\cal A}}$ for the AC/CH system is then defined as $\tilde{{\cal A}}(\tilde\varphi_{i,j}^n,\tilde\phi_{i,j}^n)^T
=(-[\nabla \cdot c(u,v)\nabla\tilde\varphi]_{i,j}^n,\frac{c(u,v)}{\rho}\tilde\phi_{i,j}^n)^T$.
With the periodic boundary conditions or the homogeneous Neumann boundary conditions, 
we present two vital formulas
\begin{equation}
\label{summation}
\begin{aligned}
&-\sum_{i=1}^{N_x}\tilde \varphi^n_{i,j}\left[(D_x \tilde c D_x)\tilde \phi\right]^{n}_{i,j}\color{black}
= \frac1 2\sum_{i=1}^{N_x}\tilde c^n_{i+\frac 12,j}[D_x^+\tilde \varphi]^{n}_{i,j} [D_x^+\tilde \phi]^{n}_{i,j}+\tilde c^n_{i-\frac 12,j}
[D_x^-\tilde \varphi]^{n}_{i,j} [D_x^-\tilde \phi]^{n}_{i,j},\\
&-\sum_{i=1}^{N_y}\tilde \varphi^n_{i,j}\left[(D_y \tilde c D_y)\tilde \phi\right]^{n}_{i,j}\color{black}
=  \frac1 2\sum_{i=1}^{N_y}\tilde c^n_{i,j+\frac 12}[D_y^+\tilde \varphi]^{n}_{i,j} [D_y^+\tilde \phi]^{n}_{i,j}+\tilde c^n_{i,j-\frac 12}
[D_y^-\tilde \varphi]^{n}_{i,j} [D_y^-\tilde \phi]^{n}_{i,j}.
   \end{aligned}
\end{equation}
\color{black}
Next, we prove \eqref{summation} in the case of homogeneous Neumann boundary conditions.
The proof  with periodic boundary conditions can be obtained similarly. 
Let us denote 
$(x_0, y_j) = \left(-\frac{1}{2}\Delta x, (j-\frac{1}{2})\Delta y\right)$ and
$(x_{N_x+1}, y_j) = ((N_x$ $+\frac{1}{2})\Delta x, (j-\frac{1}{2})\Delta y)$ be ghost points outside the computational domain. 
Assume
$\tilde \phi^n_{0,j}\approx \phi(x_0,y_i,t_n)$, $\tilde \phi^n_{N_x+1,j}\approx \phi(x_{N_x+1},y_i,t_n)$.
 The homogeneous Neumann boundary conditions in the $x$-direction are discretized as 
 \begin{equation*}
\begin{aligned}
\tilde \phi^{n}_{1,j}-\tilde \phi^{n}_{0,j}=0,\qquad
\tilde \phi^{n}_{N_x+1,j}-\tilde \phi^{n}_{N_x,j}=0.
   \end{aligned}
\end{equation*}
The first equation of \eqref{summation}  is derived as follows
\begin{equation*}
\begin{aligned}
&-\sum_{i=1}^{N_x}\tilde \phi^n_{i,j}\left[(D_x \tilde c D_x)\tilde \phi\right]^{n}_{i,j}
= 
-\sum_{i=1}^{N_x}\tilde \phi^n_{i,j}\frac{\tilde c^{n}_{i+\frac 1 2,j} \left(\tilde \phi^{n}_{i+ 1 ,j}-\tilde \phi^{n}_{i,j}\right)-
 \tilde c ^{n}_{i-\frac 1 2,j}\left(\tilde \phi^{n}_{i,j}-\tilde \phi^{n}_{i-1,j}\right)
 }{(\Delta x)^2}\\
 &\qquad=\tilde \phi^n_{1,j}\frac{
 \tilde c ^{n}_{\frac 1 2,j}\left(\tilde \phi^{n}_{1,j}-\tilde \phi^{n}_{0,j}\right)
 }{(\Delta x)^2}-\tilde \phi^n_{N_x,j}
 \frac{
 \tilde c ^{n}_{N_x+\frac 1 2,j}\left(\tilde \phi^{n}_{N_x+1,j}-\tilde \phi^{n}_{N_x,j}\right)
 }{(\Delta x)^2}\\
 &\qquad\quad
+\frac 1 2\sum_{i=1}^{N_x-1}\frac{\tilde c^{n}_{i+\frac 1 2,j} \left(\tilde \phi^{n}_{i+ 1 ,j}-\tilde \phi^{n}_{i,j}\right)^2
 }{(\Delta x)^2}+\frac 1 2\sum_{i=2}^{N_x}\frac{\tilde c^{n}_{i-\frac 1 2,j} \left(\tilde \phi^{n}_{i ,j}-\tilde \phi^{n}_{i-1,j}\right)^2
 }{(\Delta x)^2}
\\
&\qquad=\frac1 2\sum_{i=1}^{N_x}\tilde c^n_{i+\frac 12,j}[D_x^+\tilde \phi]^{n}_{i,j} [D_x^+\tilde \phi]^{n}_{i,j}+\tilde c^n_{i-\frac 12,j}
[D_x^-\tilde \phi]^{n}_{i,j} [D_x^-\tilde \phi]^{n}_{i,j}.
   \end{aligned}
\end{equation*}
The summation-by-parts formula in $y$-direction can be obtained similarly. 
\color{black}

\color{black}
It follows from the Cauchy--Schwarz inequality  and \eqref{summation} that
\color{black}
\begin{equation}
\label{neq:02}
\left\{
\begin{aligned}
&\left<\tilde \phi^n_{i,j}[{A}_{11}\tilde\phi]_{i,j}^n\right>\geq 0,\\
&\left<\tilde \varphi^n_{i,j}[{A}_{11}\tilde\phi]_{i,j}^n\right>
=\left<\tilde \phi^n_{i,j}[{A}_{11}\tilde\varphi]_{i,j}^n\right>,
\\
&\left<\tilde \varphi^n_{i,j}[{A}_{11}\tilde\phi]_{i,j}^n\right>
\leq \left< \tilde \phi^n_{i,j}[{A}_{11}\tilde\phi]_{i,j}^n\right>^{1/2}\left< 
\tilde \varphi^n_{i,j}[{A}_{11} \tilde\varphi]_{i,j}^n\right>^{1/2},
\end{aligned}
\right.
\end{equation}
where $A_{11}=-\nabla \cdot c(u,v) \nabla$.
Thus, 
we conclude that the discrete operator $\tilde{\cal A}$ for the AC/CH system is semi-positive.

With the aforementioned notations, the discretizations of the local free energy $e$ and the free energy ${\cal E}$ at time $t_n$ are respectively defined as
\begin{equation}
\label{relations}
\begin{aligned}
\left[\tilde{e}\right]_{i,j}^{n}= &e_1(\tilde u_{i,j}^n,\tilde v_{i,j}^n)+e_2(\tilde u_{i,j}^n,\tilde v_{i,j}^n) +\frac{\gamma}{2}\left(\left[(D_x^\pm\tilde{u})^2\right]^n_{i,j}+\left[(D_y^\pm\tilde{u})^2\right]^n_{i,j}\right)\\
          &+\frac{\gamma}{2}\left(\left[(D_x^\pm\tilde{v})^2\right]^n_{i,j}+\left[(D_y^\pm\tilde{v})^2\right]^n_{i,j}\right).
   \end{aligned}
\end{equation}
and
\begin{equation}
\label{freeenergy}
\begin{aligned}
\tilde{{\cal E}}^{n} :=\left<[\tilde{e}]^n_{i,j} \right> =\sum_{i=1}^{N_x}\sum_{j=1}^{N_y}[\tilde{e}]^n_{i,j} \Delta x \Delta y,
\end{aligned}
\end{equation}
where $\left<\Box_{i,j} \right>$ is defined as $\left<\Box \right>:=\sum_{i=1}^{N_x}\sum_{j=1}^{N_y}\Box_{i,j} \Delta x \Delta y$.

To obtain a full discretization scheme for the AC/CH system,  a discrete form of the variational derivative ${\mathbf{G}}$ is needed. 
By taking $\mathbf{U}(x_i, y_j, t_n) := \tilde{\mathbf{U}}^{n}_{i,j}$ and $\delta \mathbf{U}\big|_{(x_i, y_j)}:= \tilde{\mathbf{U}}^{n+1}_{i,j} - \tilde{\mathbf{U}}^{n}_{i,j}$,  
we obtain  
$
\big(e_\delta - e(\mathbf{U},$ $ \nabla \mathbf{U}^T)\big)\big|_{(x_i, y_j)} :=  [\tilde e]_{i,j}^{n+1}-
[\tilde e]_{i,j}^{n}
$.
We then choose the discrete variational derivative such that the following summation formula exactly holds
\begin{eqnarray}
\label{DVD}
\begin{aligned}
\tilde{{\cal E}}^{n+1}-\tilde{{\cal E}}^{n}&=  \left<[\tilde e]_{i,j}^{n+1}-
[\tilde e]_{i,j}^{n} \right>=\left<  \tilde{\mathbf{G}}_{i,j}  
\cdot \left(\tilde{\mathbf{U}}_{i,j}^{n+1}-\tilde{\mathbf{U}}_{i,j}^{n}\right) \right>,
\end{aligned}
\end{eqnarray}
where $\tilde{\mathbf{G}}_{i,j} $ expresses the discrete 
variational derivative. Equation (\ref{DVD})
can be viewed as a discrete form of  (\ref{condG}).
According to \eqref{neq:01} and \eqref{DVD}, we can finally reduce the discrete variational derivative
to the following form
\begin{equation*}
\tilde{\mathbf{G}} _{i,j} := \mathbf{G}_1(\tilde{\mathbf{U}}_{i,j}^{n+1}, 
\tilde{\mathbf{U}}_{i,j}^{n}) + \mathbf{G}_2(\tilde{\mathbf{U}}_{i,j}^{n+1}, \tilde{\mathbf{U}}_{i,j}^{n}) + \mathbf{G}_3([\nabla \tilde{\mathbf{U}}^T]_{i,j}^{n+1}, 
[\nabla \tilde{\mathbf{U}}^T]_{i,j}^{n}),
\end{equation*}
where $\mathbf{G}_1, \mathbf{G}_3$ are derived from $e_1$ and $e_3$ 
in local free energy, respectively.  Since both $e_1$ and $e_3$ are polynomials,
we can obtain that $\mathbf{G}_1$ and $ \mathbf{G}_3$ are also in polynomial forms as follows
\begin{equation}
\mathbf{G}_1(\tilde{\mathbf{U}}_{i,j}^{n+1}, 
\tilde{\mathbf{U}}_{i,j}^{n}) :=
\left(
\begin{array}{l}
\begin{aligned}
  -\frac\alpha 2\left(\tilde{u}_{i,j}^{n+1}+\tilde{u}_{i,j}^{n}-1\right)\\ 
  -\frac\beta 2\left(\tilde{v}_{i,j}^{n+1}+\tilde{v}_{i,j}^{n}\right)   \end{aligned}
\end{array}\right),
\label{dvds-2}
\end{equation}
\begin{equation}
\mathbf{G}_3([\nabla \tilde{\mathbf{U}}^T]_{i,j}^{n+1}, 
[\nabla \tilde{\mathbf{U}}^T]_{i,j}^{n}):=
\left(
\begin{array}{l}
\begin{aligned}
  -\frac{\gamma}{2} \left(\left[\Delta\tilde{u}\right]_{i,j}^{n+1}+\left[\Delta\tilde{u}\right]_{i,j}^{n}\right)
  \\ 
  -\frac{\gamma}{2} \left(\left[\Delta\tilde{v}\right]_{i,j}^{n+1}+\left[\Delta\tilde{v}\right]_{i,j}^{n}\right)
    \end{aligned}
\end{array}\right).
\label{dvds-1}
\end{equation}

However,  since $e_2$ is not a polynomial of $\mathbf{U}$, we \color{black} cannot  \color{black} choose $\mathbf{G}_2$  as a polynomial such that
\eqref{DVD} exactly holds.
Consider an alternative form of
$\mathbf{G}_2$ as 
\begin{equation}\label{D2}
    \mathbf{G}_2(\tilde{\mathbf{U}}_{i,j}^{n+1}, \tilde{\mathbf{U}}_{i,j}^{n})= \left(
 \begin{aligned}
   \theta\frac{\Phi(p^{n+1}_{ij})-\Phi(p^n_{ij})}{p_{ij}^{n+1}-p_{ij}^n} + \theta\frac{\Phi(q_{ij}^{n+1})-\Phi(q_{ij}^{n})}{q_{ij}^{n+1}
   -q_{ij}^{n}}\\
    \theta\frac{\Phi(p_{ij}^{n+1})-\Phi(p_{ij}^{n})}{p_{ij}^{n+1}-p_{ij}^{n}} - \theta\frac{\Phi(q_{ij}^{n+1})-\Phi(q_{ij}^{n})}{q_{ij}^{n+1}
   -q_{ij}^{n}}
   \end{aligned} \right)
\end{equation}
such that $\mathbf{G}_2(\tilde{\mathbf{U}}_{i,j}^{n+1}, \tilde{\mathbf{U}}_{i,j}^{n})\cdot\left(\tilde{\mathbf{U}}_{i,j}^{n+1}-\tilde{\mathbf{U}}_{i,j}^{n}\right)={e_2(\tilde{\mathbf{U}}_{i,j}^{n+1})-e_2(\tilde{\mathbf{U}}_{i,j}^{n})}$. 
Here $p_{ij}^{n}=\tilde{u}_{ij}^n+\tilde{v}_{ij}^n$ and $q_{ij}^n=\tilde{u}_{ij}^n-\tilde{v }_{ij}^n$.

Using the trapezoidal rule at the half-time level, 
we obtain fully discretized scheme for the AC/CH system as
\begin{equation}
\frac{\tilde{\mathbf{U}}_{i,j}^{n+1} - \tilde{\mathbf{U}}_{i,j}^{n}}{\Delta t_n} = -
\tilde{{\cal A}} \tilde{\mathbf{G}}_{i,j}.
\label{scheme-1}
\end{equation} 
\color{black}
The stability of the proposed scheme  \eqref{scheme-1} is given by the following theorem.
\color{black}
\begin{theorem}
For any given time step $\Delta t_n >0$, the numerical scheme \eqref{scheme-1}
is unconditionally energy stable and the solution 
of scheme~\eqref{scheme-1} satisfies the energy dissipative law
$
\tilde {\cal E}^{n+1}\leq \tilde {\cal E}^n.
$
\end{theorem}
\color{black}
\begin{proof}
Similar to the proof of Theorem 2 in \cite{Shin}, according to \eqref{neq:02}, \eqref{DVD},  and \eqref{scheme-1}, 
the discrete free energy induced by scheme \eqref{scheme-1} satisfies
\begin{equation}
\begin{aligned}
   \frac{\tilde {\cal E}^{n+1}- \tilde {\cal E}^n}{\Delta t_n}
   =\left<
 \tilde{\mathbf{G}}_{i,j}  \cdot \frac{\tilde{\mathbf{U}}_{i,j}^{n+1} - \tilde{\mathbf{U}}_{i,j}^{n}}{\Delta t_n}\right>=-\left<
 \tilde{\mathbf{G}}_{i,j}  \cdot 
\tilde{{\cal A}} \tilde{\mathbf{G}}_{i,j}  \right>\leq 0. 
   \end{aligned}
   \label{DF}
\end{equation}
\end{proof}
\color{black}

\color{black}
Since the total free energy functional $\cal E$ is non-convex,
the existence and uniqueness of the solution $U^{n+1}$ for system \eqref{scheme-1} (especially for large values of $\Delta t$) are
not immediate. Nonetheless, the proof of Theorem 2.1 did not require a unique solution to \eqref{scheme-1}.
In fact, 
\color{black}
even if one could prove that scheme \eqref{scheme-1} is  unconditionally energy stable,
it is numerically unstable when applied  to  solve the AC/CH system.
The reason is that the numerical computation 
of fraction ${\xi_1}/{\xi_2}$ is unstable and inaccurate when 
$\xi_2$ is close to zero. Unfortunately, in the AC/CH system, $\delta p_{i,j}=p_{ij}^{n+1}-p_{ij}^n$  or $q_{ij}^{n+1}-q_{ij}^n$
is often close to zero,  which indicates that the numerical calculation of $\mathbf{G}_2$ by \eqref{D2}
is numerically unstable and inaccurate in this situation. As a result, the scheme \eqref{scheme-1} may lead to an inaccurate or non-physical solution 
with the existence of the complicated function $e_2$.
To overcome this difficulty, we propose an approach to calculate  $\mathbf{G}_2$, which is numerically stable and highly accurate. 
In this approach, 
$\Phi(p_{ij}^{n+1})$ and $\Phi(p_{ij}^n)$ are approximately calculated by Taylor expansion
at  the point $\bar {p}_{i,j}:= \frac{p_{i,j}^{n+1}+p_{i,j}^n }2$ for very small $\delta p_{i,j}$. 
Analogous methods can be applied to terms related to $\Phi(q_{ij}^{n+1})$ and $\Phi(q_{ij}^n)$.
The accuracy of the approximation is guaranteed  by the following \color{black} lemma, which can be verified directly from the Taylor expansion. \color{black} 
\begin{lemma}\label{lemma2}
Let us assume $f(x,\sigma) = \frac{(x+\sigma)ln(x+\sigma)-(x-\sigma)ln(x-\sigma)}{2\sigma}$. 
If $\left| \frac \sigma x\right|\ll 1$, we have the following expansion 
\begin{equation}
\label{taylor}
\begin{aligned}
f(x,\sigma) &= \textnormal{ln}(x) - \frac{1}{6} \left(\frac{\sigma}{x}\right)^2 - \cdots 
- \frac{1}{(2S+1)2S} \left(\frac{\sigma}{x}\right)^{2S}+ {\cal O}\left(\frac{1}{S^2}\left(\frac{\sigma}{x}\right)^{2S+2}\right)\\
& := f_S(x,\sigma) + {\cal R}_S.
\end{aligned}
\end{equation}
Here the truncation error ${\cal R}_S\rightarrow 0$ as $S\rightarrow \infty $.

\end{lemma}

With the Taylor expansion \eqref{taylor}, 
$\mathbf{G}_2$ is approximated as
\begin{equation}\label{D2-1}
    \mathbf{G}^S_2(\tilde{\mathbf{U}}_{i,j}^{n+1}, \tilde{\mathbf{U}}_{i,j}^{n})= \left(
 \begin{aligned}
   \theta\phi_{i,j}^S + \theta\psi_{i,j}^S\\
   \theta\phi_{i,j}^S - \theta\psi_{i,j}^S
   \end{aligned} \right):= \mathbf{G}_2(\tilde{\mathbf{U}}_{i,j}^{n+1}, \tilde{\mathbf{U}}_{i,j}^{n})-{{\cal R}}^S_{i,j},
\end{equation}
where  ${\phi}_{i,j}^S$ and  ${\psi}_{i,j}^S$ are given as follows
\begin{equation}
\begin{array}{l}
\begin{aligned}
{ \phi}_{i,j}^S:= f_S(\zeta_p, \sigma_p) + f_S(1-\zeta_p, -\sigma_p),\\
{ \psi}_{i,j}^S:=f_S(\zeta_q, \sigma_q) + f_S(1-\zeta_q, -\sigma_q)  .
  \end{aligned}
\end{array}
\end{equation}
Here $\zeta_p=\frac{{ p}_{i,j}^{n+1}+{p}_{i,j}^{n}}{2}$, $\sigma_p=\frac{{ p}_{i,j}^{n+1}-{ p}_{i,j}^{n}}{2}$, $\zeta_q=\frac{{ q}_{i,j}^{n+1}+{ q}_{i,j}^{n}}{2}$, 
 and $\sigma_q=\frac{{ q}_{i,j}^{n+1}-{ q}_{i,j}^{n}}{2}$.
Since $\left|\frac{\sigma_p}{\zeta_p}\right|\ll 1$, $\left|\frac{\sigma_p}{1-\zeta_p}\right|\ll 1$, $\left|\frac{\sigma_q}{\zeta_q}\right|\ll 1$, and 
$\left|\frac{\sigma_q}{1-\zeta_q}\right|\ll1$, one can obtain ${{\cal R}}^S_{i,j}\rightarrow 0$ as $S\rightarrow\infty$ from Lemma \ref{lemma2}.

With the approximation, 
 scheme \eqref{scheme-1} for the AC/CH system is replaced with
\begin{equation}
\label{NSOA-1}
\frac{\tilde{\mathbf{U}}_{i,j}^{n+1} - \tilde{\mathbf{U}}_{i,j}^{n}}{\Delta t_n} = -
\tilde{{\cal A}} \tilde{\mathbf{G}}_{i,j}^S ,
\end{equation}
where the approximate discrete variational derivative is defined as
\begin{equation}
\tilde{\mathbf{G}}_{i,j}^S := \mathbf{G}_1(\tilde{\mathbf{U}}_{i,j}^{n+1}, 
\tilde{\mathbf{U}}_{i,j}^{n}) + \mathbf{G}^S_2(\tilde{\mathbf{U}}_{i,j}^{n+1}, \tilde{\mathbf{U}}_{i,j}^{n}) + \mathbf{G}_3([\nabla \tilde{\mathbf{U}}^T]_{i,j}^{n+1}, 
[\nabla \tilde{\mathbf{U}}^T]_{i,j}^{n}).
\end{equation}
It should be noted that the solution obtained by using scheme \eqref{NSOA-1} is usually different from the one by scheme \eqref{scheme-1} due to 
the different approximation made by the two schemes.
For  simplicity, we still denote  the solution  of scheme \eqref{NSOA-1} as $\tilde{\mathbf{U}}_{i,j}^{n+1}$.
According to \eqref{DVD} and \eqref{D2-1}, we have 
\begin{eqnarray}
\label{eq:19}
\tilde{{\cal E}}^{n+1}-\tilde{{\cal E}}^{n}=\left<  \tilde{\mathbf{G}}_{i,j}  
\cdot \left(\tilde{\mathbf{U}}_{i,j}^{n+1}-\tilde{\mathbf{U}}_{i,j}^{n}\right) \right>
=\left< \left( \tilde{\mathbf{G}}_{i,j}^{S} +{\cal R}^S_{i,j}\right) 
\cdot \left(\tilde{\mathbf{U}}_{i,j}^{n+1}-\tilde{\mathbf{U}}_{i,j}^{n}\right) \right>.~~~~
\end{eqnarray}
The stability of the proposed scheme  \eqref{NSOA-1} is given by the following theorem.

\begin{theorem}\label{theorem1}
Given 
 $\Delta t_n>0$, scheme~\eqref{NSOA-1} is unconditionally energy stable and the solution 
 satisfies the following energy dissipation relationship with a cut off error
\begin{equation}
\hat {\cal E}^{n+1}\leq \hat {\cal E}^n+\frac{\Delta t_n}{4}\left<
{\cal R}^S_{i,j} \cdot \tilde{{\cal A}} {\cal R}^S_{i,j}\right>.
\label{diss-1}
\end{equation}
Furthermore,  there exists an integer $S_0$ such that  
\begin{equation}
\hat {\cal E}^{n+1}\leq \hat {\cal E}^n,
\label{diss11}
\end{equation}
holds for any $S>S_0$, which implies that the solution of scheme~\eqref{NSOA-1} obeys the energy dissipative law.
\end{theorem}
\begin{proof}
According to  \eqref{NSOA-1}, and \eqref{eq:19},
the discrete free energy decided by scheme \eqref{NSOA-1} satisfies
\begin{equation}\label{eq::27}
\begin{aligned}
   &\frac{\tilde {\cal E}^{n+1}- \tilde {\cal E}^n}{\Delta t_n}
   = \left<
 \tilde{\mathbf{G}}_{i,j} \cdot \frac{\tilde{\mathbf{U}}_{i,j}^{n+1} - \tilde{\mathbf{U}}_{i,j}^{n}}{\Delta t_n} \right>=-\left<
\tilde{\mathbf{G}}_{i,j} \cdot 
\tilde{{\cal A}} \tilde{\mathbf{G}}^S_{i,j}\right>\\
&\qquad=-
\left< \left( \tilde{\mathbf{G}}_{i,j} -\frac 1 2{\cal R}^S_{i,j}\right)  \cdot 
\tilde{{\cal A}} \left(\tilde{\mathbf{G}}_{i,j} -\frac 1 2{\cal R}^S_{i,j}\right) -\frac{1}{4} {\cal R}^S_{i,j} \cdot \tilde{{\cal A}} {\cal R}^S_{i,j}\right>\\
&\qquad  \leq \left<
\frac{1}{4} {\cal R}^S_{i,j} \cdot \tilde{{\cal A}} {\cal R}^S_{i,j} \right>, 
   \end{aligned}
\end{equation}
which completes the proof of \eqref{diss-1}.
Furthermore, from \eqref{eq::27}, we have
\begin{equation}\label{eq::27-1}
\begin{aligned}
   &\frac{\tilde {\cal E}^{n+1}- \tilde {\cal E}^n}{\Delta t_n}
   =-\left<
\tilde{\mathbf{G}}_{i,j}  \cdot 
\tilde{{\cal A}} \tilde{\mathbf{G}}_{i,j}^S\right> =-
\left<
\tilde{\mathbf{G}}_{i,j}\cdot 
\tilde{{\cal A}} \tilde{\mathbf{G}}_{i,j} -
\tilde{\mathbf{G}}_{i,j} \cdot 
\tilde{{\cal A}} {\cal R}_{i,j}^S\right> \\
&\qquad \leq-\left<
\tilde{\mathbf{G}}_{i,j}  \cdot 
\tilde{{\cal A}}\tilde{\mathbf{G}}_{i,j}\right> +\left<
\tilde{\mathbf{G}}_{i,j}  \cdot 
\tilde{{\cal A}} \tilde{\mathbf{G}}_{i,j}\right>^{1/2}\left<
{\cal R}^S_{i,j} \cdot \tilde{{\cal A}} {\cal R}^S_{i,j} \right>^{1/2}.
   \end{aligned}
\end{equation}
If ~$\left<
\tilde{\mathbf{G}}_{i,j}  \cdot 
\tilde{{\cal A}} \tilde{\mathbf{G}}_{i,j}\right>=0$,
 we have $\tilde {\cal E}^{n+1}\leq \tilde {\cal E}^n$. Otherwise, we set
 $\left<
\tilde{\mathbf{G}}_{i,j}  \cdot 
\tilde{{\cal A}} \tilde{\mathbf{G}}_{i,j}\right>=\varrho>0$.
Since ${\cal R}^S_{i,j}\rightarrow 0$ as $S\rightarrow\infty$, there exists an integer $S_0>0$ such that
$\left<
{\cal R}^S_{i,j} \cdot \tilde{{\cal A}} {\cal R}^S_{i,j}\right> \leq \varrho/2$ holds for any $S>S_0$.
Thus, we have $\tilde {\cal E}^{n+1}\leq \tilde {\cal E}^n$, which completes the proof of \eqref{diss11}.
\end{proof}

It is worth noting that the method proposed here can be generalized to a much broader range of \color{black} phase field \color{black} equations,
which can be used to model a gradient flow or more general dissipation mechanism. Following the procedure of this section,
one can further prove that the constructed scheme for the general phase field system is again unconditionally energy
stable. The framework presented in this paper can deal with the local free energy with any complex formulations   
arising from the real applications. \color{black}
This provides an alternative approach to well-known methods such as convex splitting \color{black}
 \cite{Baskaran, Elliott, Eyre, Shen1},  invariant energy quadratization \cite{XYang, JZhao},  and scalar auxiliary variable \cite{Shen3}.

\section{Parallel domain decomposition solver}
\subsection{ Newton--Krylov--Schwarz solver}
Denote $\mathbf{X}^{n}=(u^{n}_{0,0}, v^{n}_{0,0},$ $ u^{n}_{1,0},$ $ v^{n}_{1,0},$ $  u^{n}_{2,0}, v^{n}_{2,0}, \cdots)^{T}$.
By discretizing the AC/CH system with the proposed energy stable scheme \eqref{NSOA-1}, 
a discrete nonlinear system ${\cal F}(\mathbf{X}^{n})=0$ is  
constructed and solved at each time step. We omit the superscript $n$ in the 
remainder of the subsection. We solve the nonlinear system on a parallel supercomputer 
by adopting a NKS type algorithm \cite{NKS}. The NKS algorithm consists of three 
important components: 1) an inexact Newton method as the outer iteration; 2) a Krylov method as an 
inner iteration for the linear Jacobian system at each Newton iteration; and 3) a Schwarz preconditioner 
to improve the convergence of the Krylov method.

At each time step, the solution of the previous time step is used as  the initial guess 
for the Newton iteration. At the $(m+1)$-th iteration of the inexact Newton method, the new solution 
$\mathbf{X}_{m+1}$ is obtained from the current solution $\mathbf{X}_{m}$ through
\begin{equation}
\label{newton}
\mathbf{X}_{m+1}=\mathbf{X}_{m} +\lambda_m\mathbf{S}_m,\ \ \ m=0,1,\cdots .
\end{equation}
 Here $\lambda_m$ is the step length determined by a line search procedure \cite{NMFU}, and 
$\mathbf{S}_m$ is the search direction obtained by solving a Jacobian system. 
The stopping condition for the Newton iteration \eqref{newton} is 
\begin{equation}
\|{\cal F}(\mathbf{X}_{m+1})\| \leq \max\{\varepsilon_r
\|{\cal F}(\mathbf{X}_0)\|, \varepsilon_a\},
\end{equation}
where $\varepsilon_r, \varepsilon_a \geq 0$ are the relative and absolute tolerances for the 
nonlinear iteration, respectively.

In \eqref{newton}, the search direction $\mathbf{S}_m$ is obtained by approximately solving 
the following right-preconditioned linear Jacobian system
\begin{equation}
\label{hjacobi}
J_mH_m^{-1}(H_m\mathbf{S}_m) = -{\cal F}(\mathbf{X}_m),
\end{equation}
where $J_m=\frac{\partial {\cal F}(\mathbf{X}_m)}{\partial \mathbf{X}_m}$ is the Jacobian matrix, 
and $H_m^{-1}$ is the additive Schwarz type preconditioner. 
In our study, a restarted Generalized Minimal Residual (GMRES) method \cite{gmres} is applied to approximately 
solve the right-preconditioned linear system \eqref{hjacobi}
until the linear residual $\mathbf{r}_m = J_m\mathbf{S}_m + {\cal F}(\mathbf{X}_m)$ satisfies 
the stopping condition
$$\|\mathbf{r}_m\| \leq  \textnormal{max} \{\xi_r\|{\cal F}(\mathbf{X}_m)\|, \xi_a\},$$
where $\xi_r, \xi_\alpha \geq 0$ are the relative and absolute tolerances for the linear iteration, respectively. 
In the GMRES method, the additive Schwarz type preconditioner $H_m^{-1}$ is the key to 
the success of the linear solver. To define $H_m^{-1}$, we first partition the computational domain $\Omega$ 
into $np$ non-overlapping 
subdomains $\Omega_k,\ (k = 1, 2, \cdots, np)$, then extend each subdomain by $\delta$ mesh 
layers to form an overlapping decomposition $\Omega =\cup_{k=1}^{np}\Omega_k^{\delta}$.

The classical additive Schwarz preconditioner 
\cite{DDA} is defined as
\begin{equation}
H_m^{-1}(\delta\delta)=\sum_{k=1}^{np}(R_k^{\delta})^T \textnormal{inv}(A^m_k)R_k^{\delta}.
\label{invM}
\end{equation}
Here the restriction matrix $R_k^{\delta}$ maps a vector to a new one that is defined in the 
subdomain $\Omega_k^{\delta}$, by discarding the components outside $\Omega_k^{\delta}$; 
the extension  matrix $(R_k^{\delta})^T$ maps a vector defined in the 
subdomain $\Omega_k^{\delta}$ to a new one that is defined in the whole domain, by putting zeros at the components 
outside $\Omega_k^{\delta}$. 
\color{black}In \eqref{invM}, \color{black} $A^m_k=R_k^{\delta}J_m(R_k^{\delta})^T$ is the subdomain matrix.
We calculate the matrix-vector multiplication with $\textnormal{inv}(A^m_k)$ by a sparse LU factorization or 
incomplete LU (ILU) factoriztion.

There are two popular modifications of the AS preconditioner that may have some
potential advantages,  the left restricted additive Schwarz (left-RAS, \cite{cai99sisc}) 
preconditioner and the right restricted  additive Schwarz (right-RAS, \cite{cai03sinum_ash}) 
preconditioner.
Compared to the classical AS preconditioner, the communication in the two restricted 
versions is reduced approximately by half because only one side of restriction or extension step requires 
communication. Many experiment results have shown  that the restricted
Schwarz preconditioners is generally superior to classic AS preconditioners \cite{cai99sisc, PFCDVD}. This may further improve 
the performance of the preconditioner.

\subsection{An adaptive time stepping strategy}

Theorem 2.3 shows the unconditional stability property of the implicit scheme \eqref{NSOA-1}. 
But an abrupt increase of the time step size is adverse for keeping the computational accuracy. 
Numerical experiments show that simulations with a large constant time step may produce nonphysical solutions 
\cite{zhang2013}. This is because the AC/CH system contains 
multiple time scales that may vary in orders of magnitude during the phase separation and order-disorder transitions. Therefore an adaptive control of the time step size is necessary, 
in which the time step size is selected based on the desired solution accuracy and the dynamic 
features of the system. 

To deal with the multiple time scales, we begin with the introduction of  the adaptive time stepping strategies described in \cite{PFCDVD, zhang2013}, in which 
the initial time step size $\Delta t_0$ is set as $\Delta t_{\min}$ and the time step size at the $(n+1)^{\mathrm{th}} $ time step is predicted to
\begin{equation}
\label{sencondt}
\tilde{\Delta} t_n= \max \left(\Delta t_{\min},\frac{\Delta t_{\max}}{\sqrt{1+\eta \mathbf{X}_d'(t_n)^2}} \right),
\end{equation}
where $\mathbf{X}_d'(t_n) =\frac{\|\mathbf{X}^{n}-\mathbf{X}^{n-1}\|}{\Delta t_{n-1}}$ 
corresponds to the change rate of numerical solutions on the two previous time steps, 
and $\eta$ is a positive pre-chosen parameter. 
Then we use the NKS algorithm  to solve the discrete system \eqref{NSOA-1} with the predicted time step size $\tilde{\Delta} t_n$. If the NKS solver diverges 
with the currently predicted time step size $\tilde{\Delta} t_n$, a smaller predicted time step size $\tilde{\Delta} t_n:=\tilde{\Delta} t_n/\sqrt 2 $ is chosen to restart the NKS algorithm.
The loop is broken down and the time step size $\Delta t_n$ is set to be $\tilde{\Delta} t_n$ until the NKS solver converges. 
In \eqref{sencondt}, $\Delta t_{\max}$ and $\Delta t_{\min}$  are defined as the upper and lower bounds of the time step size, 
namely $\Delta t_{\min}\leq\Delta t_n \leq \Delta t_{\max}$.
However, in the simulations of the AC/CH system, a directly using of the adaptive time stepping strategy \eqref{sencondt}
with a pre-chosen parameter $\eta$ will be low efficient. To improve, we
 initially set $\eta$ as a relatively small value, and then double the value of $\eta$  when the NKS solver diverges. 
Numerical simulations carried out in the Section 4 show the efficiency of the modification.

\section{Numerical experiments}
In this section, we investigate the numerical behavior and parallel performance of the proposed algorithm 
for the AC/CH system~\eqref{eq:1}. We carry out several two and three dimensional 
tests to validate the discretization 
of the proposed algorithm. Various performance related parameters in the NKS algorithm are 
 studied as well. 
We mainly focus on: 1) the verification of the accuracy of the proposed energy stable method, 2) a comparison of different preconditioners and subdomain solvers, 
3) the performance of the adaptive time stepping strategy, and 4) the parallel scalability of the proposed algorithm.

The numerical experiments are performed on the Sunway TaihuLight supercomputer, 
ranking the third place in the TOP--500 list as of November, 2019.
The computing power of TaihuLight is provided by a Chinese  homegrown many-core SW26010 CPU \cite{Fu2016}, 
in which we only enable one core per CPU for the current study. The algorithm for the AC/CH system is implemented on top of the Portable, Extensible Toolkits for 
Scientific computations (PETSc, \cite{petsc}) library. In the approximation scheme \eqref{taylor}, 
we set $S=10$  such that the error coming from the Taylor approximation can 
be ignored. The stopping conditions for the nonlinear and linear 
iterations are set follows. 
\begin{itemize}
\setlength{\itemsep}{0pt}
\item The relative tolerance for the nonlinear iteration: $\varepsilon_r = 1\times10^{-6}$.
\item The absolute tolerance for the nonlinear iteration: $\varepsilon_a = 1\times10^{-13}$.
\item The relative tolerance for the linear iteration: $\xi_r = 1\times10^{-8}$.
\item The absolute tolerance for the linear iteration: $\xi_a = 1\times10^{-14}$.
\end{itemize}

\subsection{Validation of the energy stable scheme}
~~

\vspace{0.2cm}
\noindent A. Two dimensional tests
\vspace{0.2cm}

In order to study the convergence behavior of the proposed energy stable scheme (2.20), we consider a two dimensional problem with periodic boundary conditions and the following initial conditions
 \begin{equation}
 \begin{aligned}
 &~u^{(0)}(x,y) = 0.4\left(\sin^2\left(2\pi x\right)+\cos^2\left(2\pi y\right)\right) + 0.1,\ \ (x,y)\in\Omega,\\
 &~v^{(0)}(x,y) = 0.4\left(\sin^2\left(2\pi x\right)-\cos^2\left(2\pi y\right)\right),\ \ (x,y)\in\Omega,
 \end{aligned}
 \end{equation}
where $\Omega=[0,1]^2$. The parameters are set as: $\alpha = 4$, $\beta = 2$, $\gamma=0.005$, $\theta = 0.1$, $\rho =0.001$. Since the exact solutions $u, v$ of the AC/CH system are unknown, the numerical solutions on a  
fine mesh $1,024 \times 1,024$ with small time step size $\Delta t = 5\times 10^{-5}$ are taken as the reference
solutions $\bar{u}^n_{i,j}, \bar{v}^n_{i,j}$. 
We define the relative $l_2$ error between the numerical solutions and the reference solutions as 
\begin{equation}
l_2 = \left(\frac{\sum_{i,j}\left((\tilde{u}^n_{i,j}-\bar{u}^n_{i,j})^2 + (\tilde v^n_{i,j}-\bar{v}^n_{i,j})^2\right)}{\sum_{i,j}((\bar{u}^n_{i,j})^2+(\bar{v}^n_{i,j})^2) }\right)^{\frac{1}{2}}.
\end{equation}
We plot the $l_2$ errors of the numerical solutions with respect to the changes of the spatial mesh resolution and the time step size in Fig.~\ref{ex1:accuracy} (a) and 
Fig.~\ref{ex1:accuracy} (b), respectively. As shown in the figures, the energy stable scheme~\eqref{NSOA-1} exhibits second-order accuracy in both space and time.

 \begin{figure}[!h]
\begin{center}
\qquad\scriptsize{(a)}\qquad\qquad\qquad\qquad\qquad\qquad\qquad\qquad\qquad
\qquad\qquad\qquad\scriptsize{(b)}\\
{\includegraphics[width=0.495\textwidth]{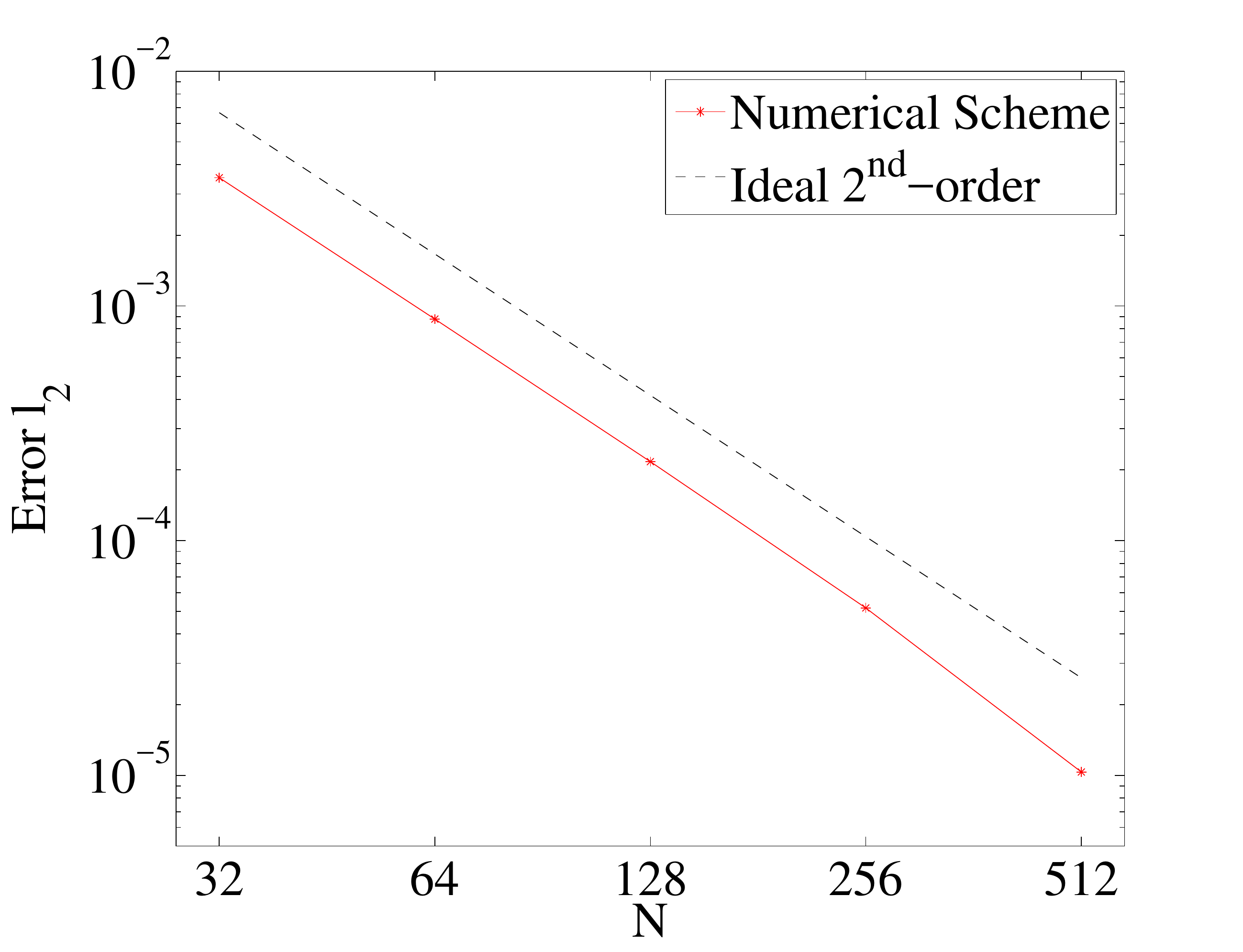}}
{\includegraphics[width=0.495\textwidth]{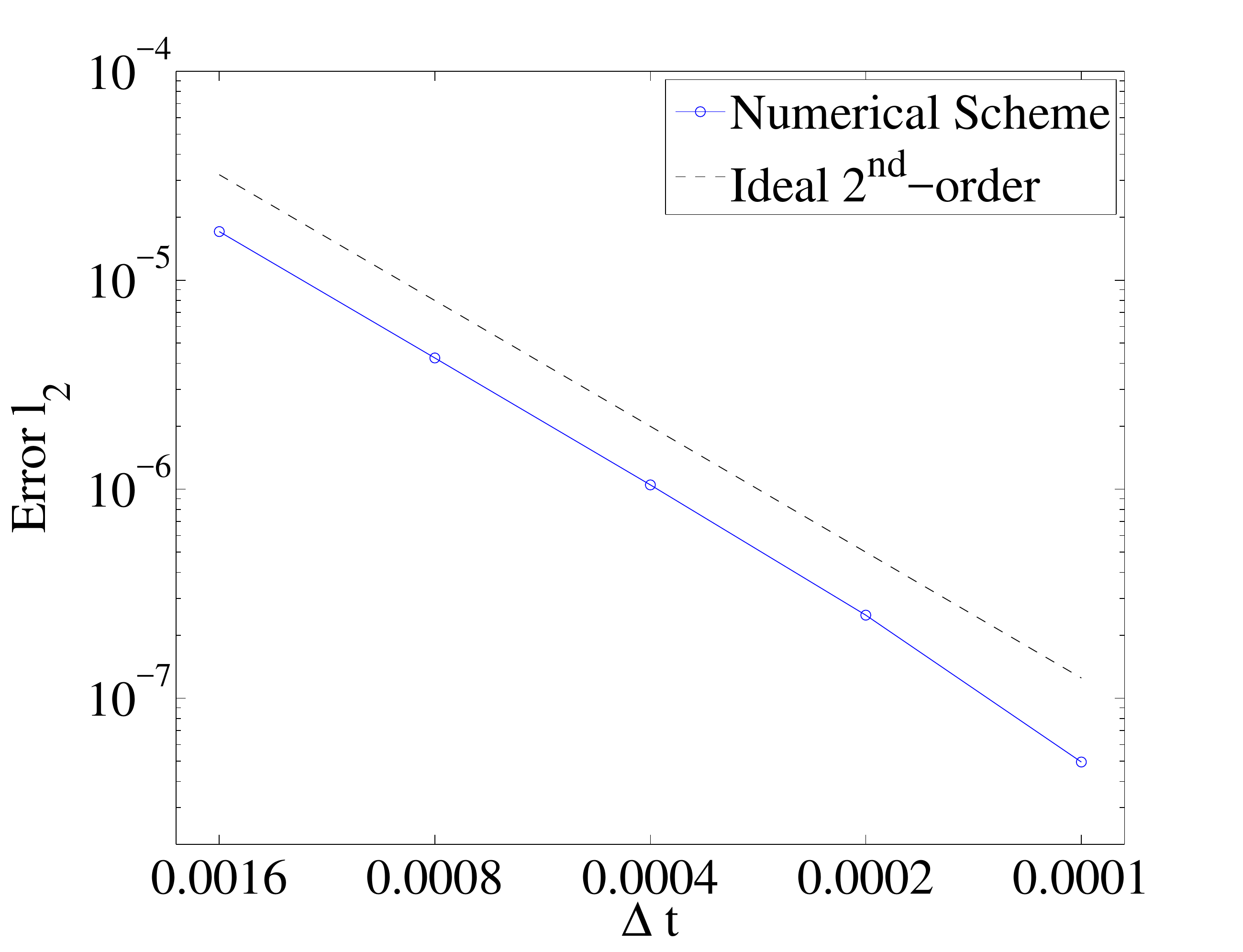}}
\end{center}
\caption{Convergence rates of the energy stable scheme in space (a) and time (b). The time step sizes $\Delta t$ of all simulations in (a) are fixed as $5\times 10^{-5}$, and
the spatial mesh $1,024\times 1,024$ is used in (b) for all tests. \color{black} In (a), we take $N_x=N_y=N$. \color{black} }
\label{ex1:accuracy} 
\end{figure}

Next we consider a two dimensional test case previously  studied in \cite{Xia1}. The initial condition 
is set as $(u^{(0)}, v^{(0)}) = (0.55 + \delta_u, \delta_v)$, 
where 􏰋$\delta_u$ and $\delta_v$ are uniform random distributions in $-0.05$ to $0.05$. 
The parameters are set to  $\alpha = 4$, $\beta = 2$, $\gamma=0.005$, $\theta = 0.1$, $\rho =0.001$, and $S=10$. 
The computational domain is $[0,1]^2$, and we run the test case on a $128\times 128$ mesh with an initial time step 
$\Delta t_1 = 10^{-4}$. The time step size is then adaptively controlled with $\eta=10^4$, $\Delta t_{\min}=10^{-4}$, and 
$\Delta t_{\max} = 10$. 

\begin{figure}[!t]
\begin{center}
\scriptsize{(a1) $t = 1$}\qquad\qquad\qquad\qquad\scriptsize{(a2) $t = 3$}
\qquad\qquad\qquad\qquad\scriptsize{(a3) $t = 6$}
\qquad\qquad\qquad\qquad\scriptsize{(a4) $t = 10,000$}\\
{\includegraphics[width=0.03\textwidth]{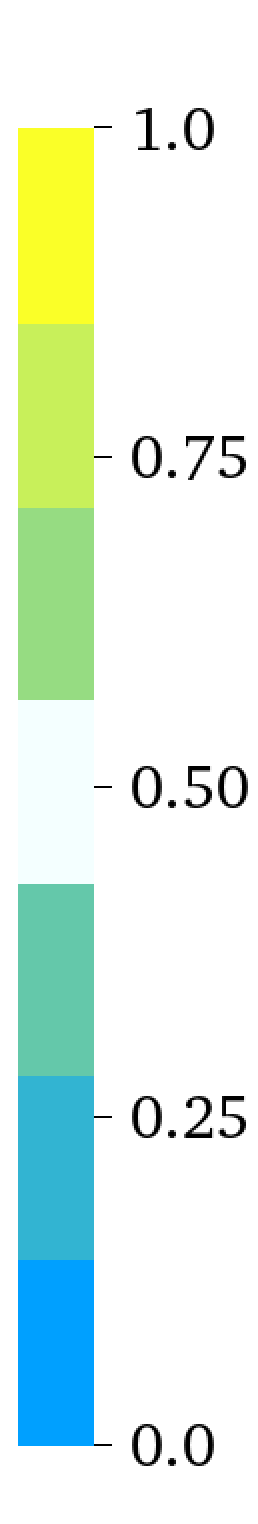}}
{\includegraphics[width=0.23\textwidth]{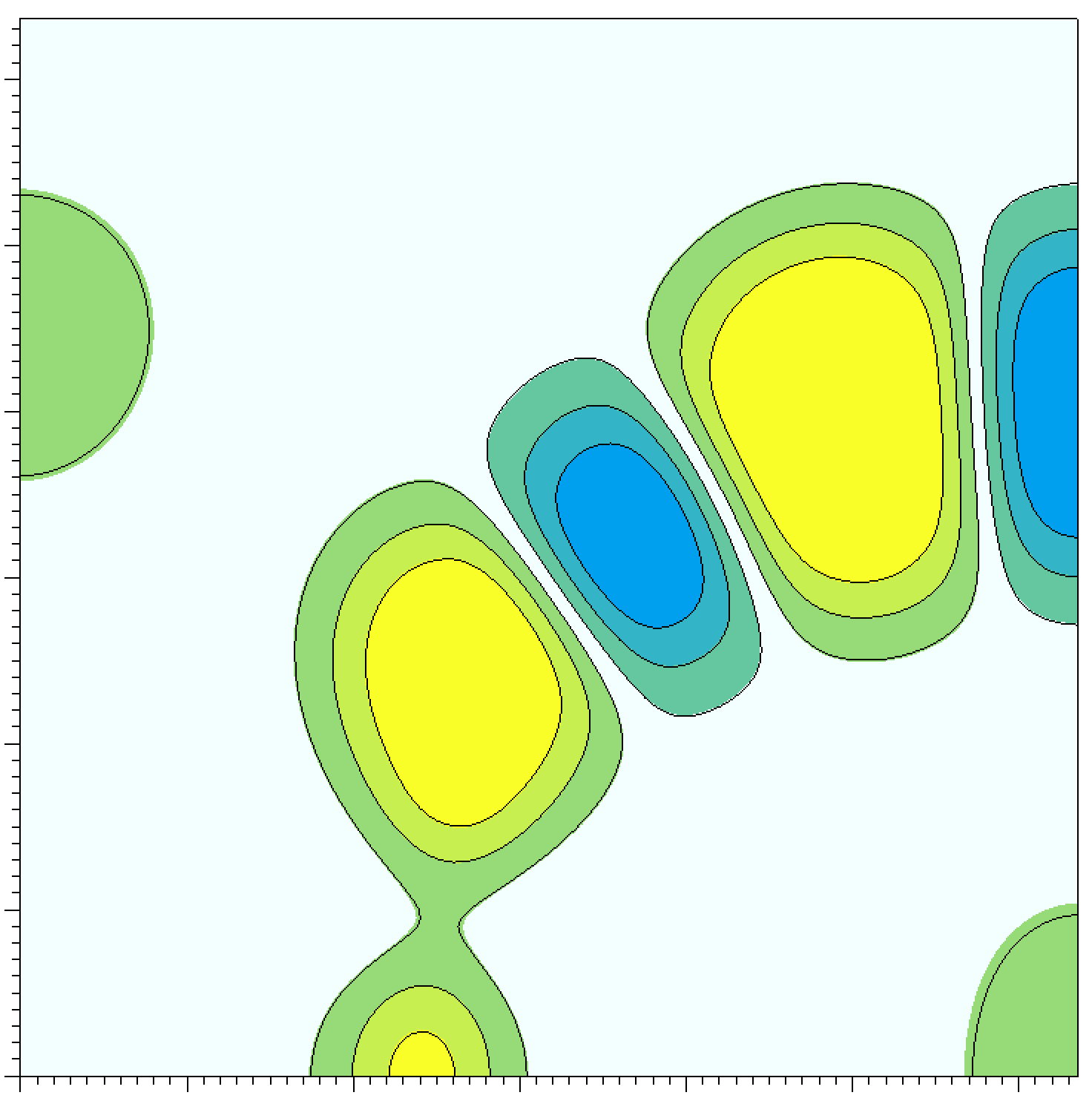}}
{\includegraphics[width=0.23\textwidth]{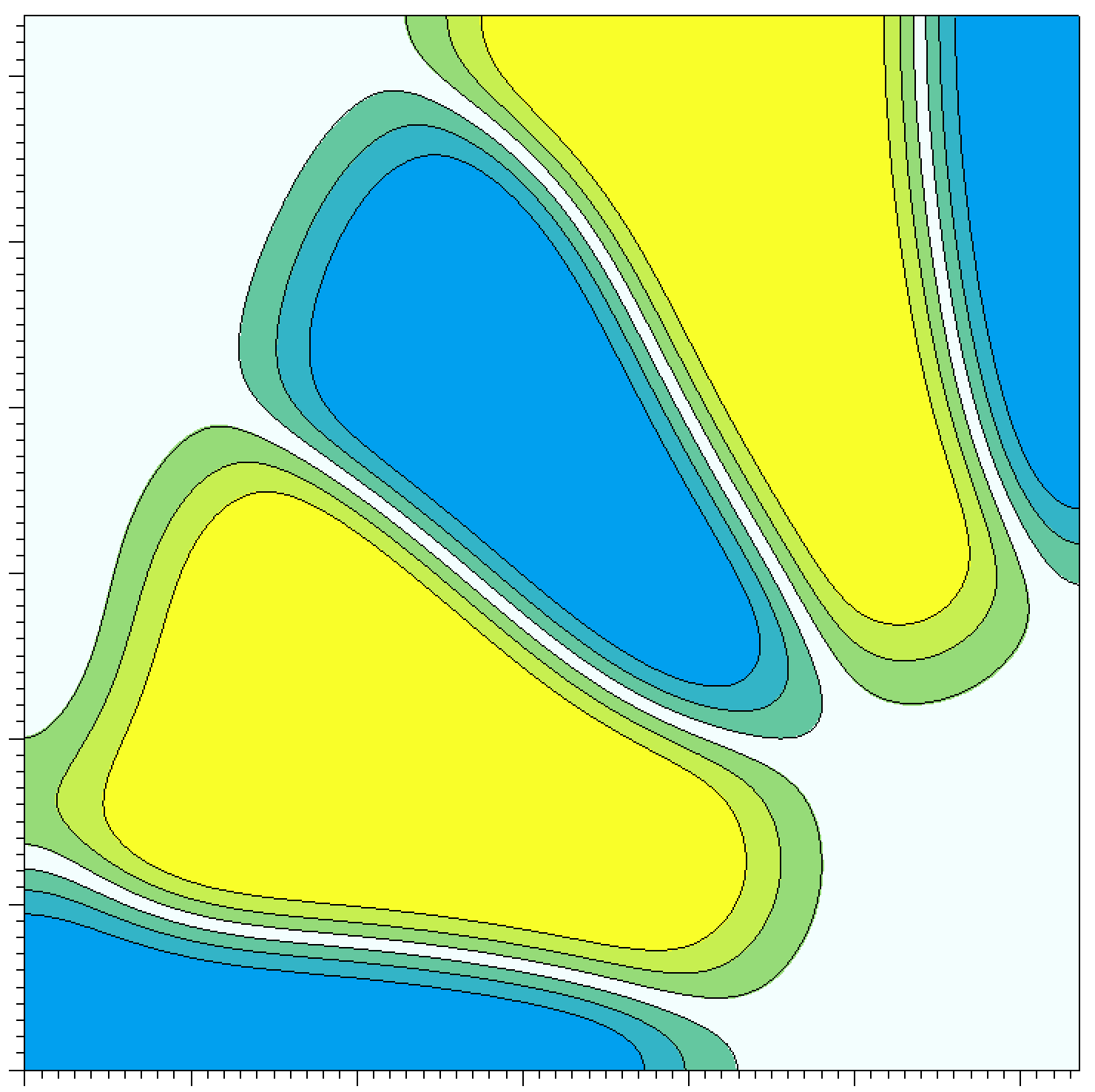}}
{\includegraphics[width=0.23\textwidth]{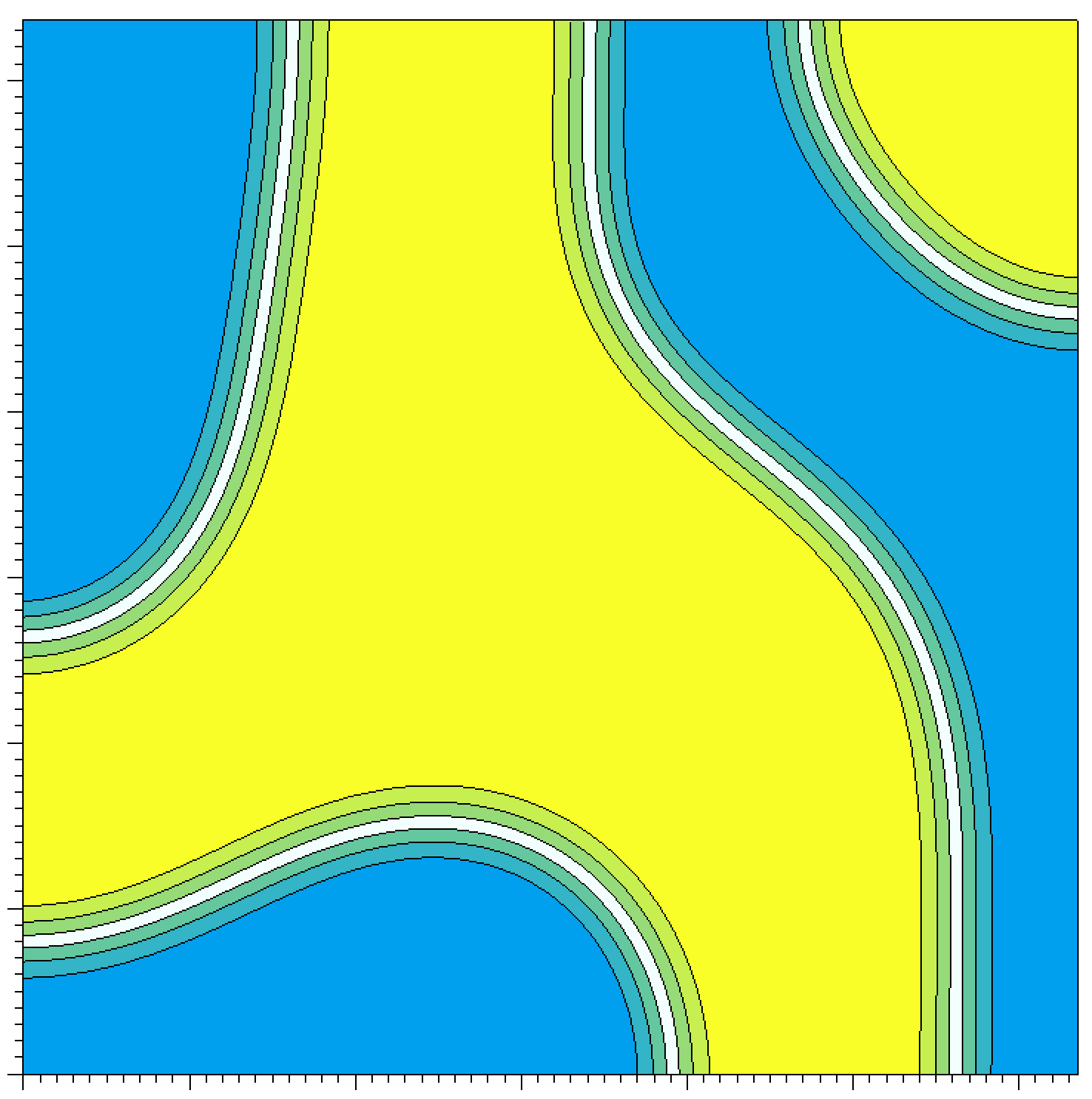}}
{\includegraphics[width=0.23\textwidth]{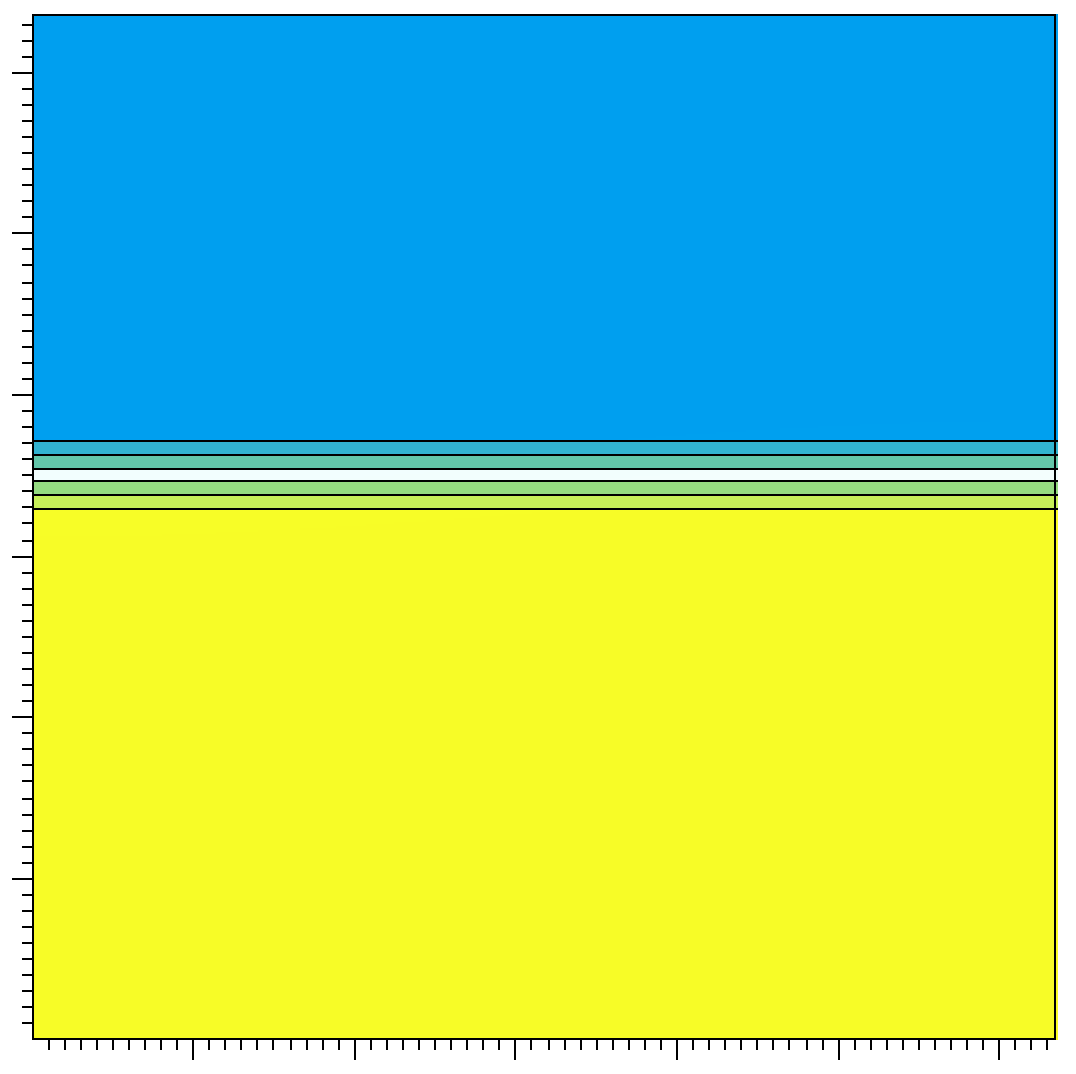}}
\\~~\\
\scriptsize{(a5) $t = 1$}\qquad\qquad\qquad\qquad\scriptsize{(a6) $t = 3$}
\qquad\qquad\qquad\qquad\scriptsize{(a7) $t = 6$}
\qquad\qquad\qquad\qquad\scriptsize{(a8) $t = 10,000$}\\
{\includegraphics[width=0.035\textwidth]{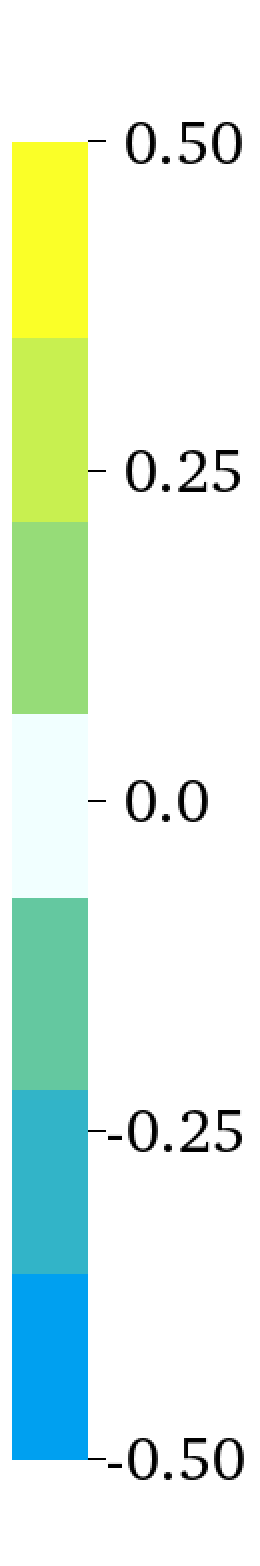}}
{\includegraphics[width=0.23\textwidth]{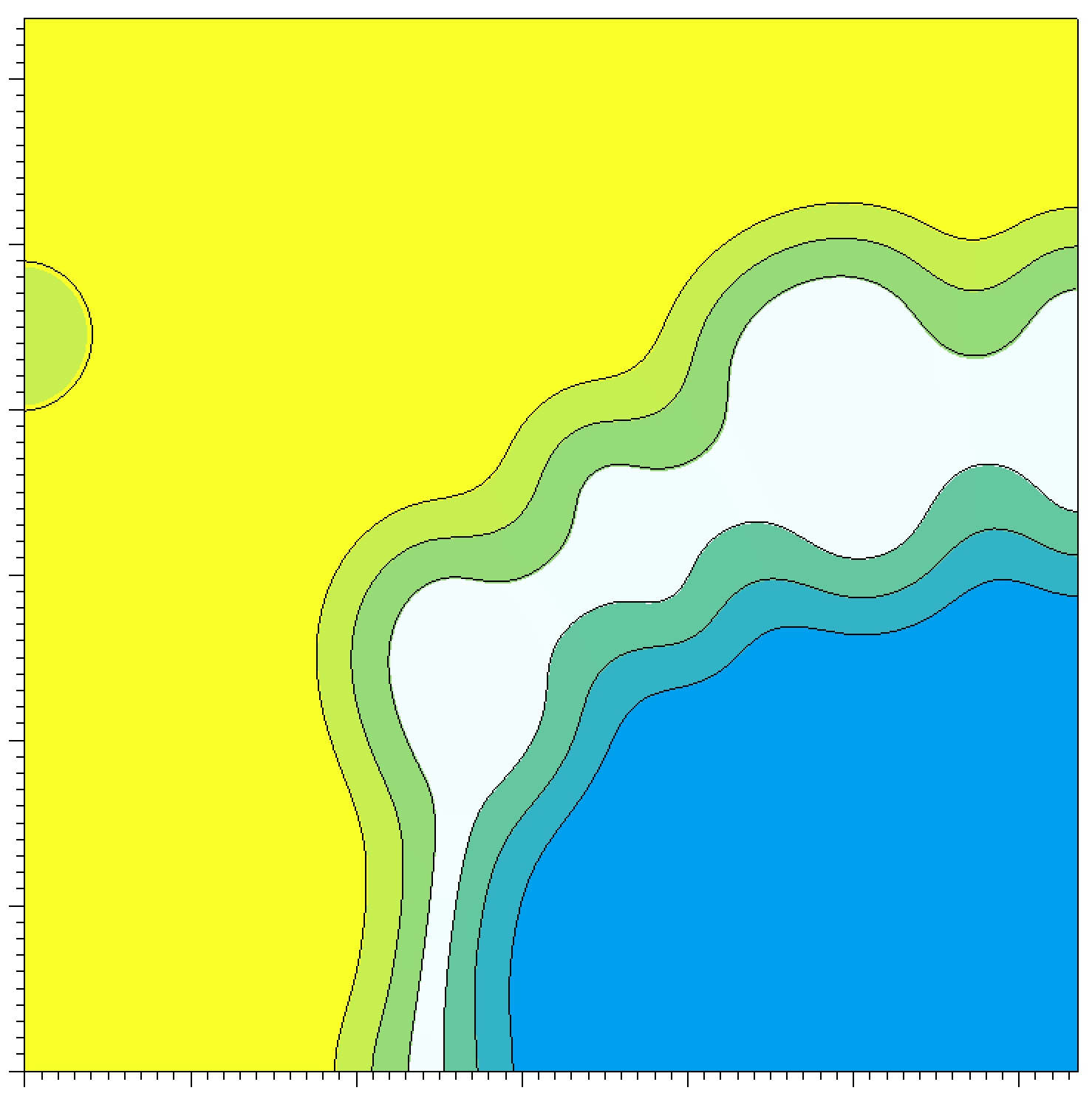}}
{\includegraphics[width=0.23\textwidth]{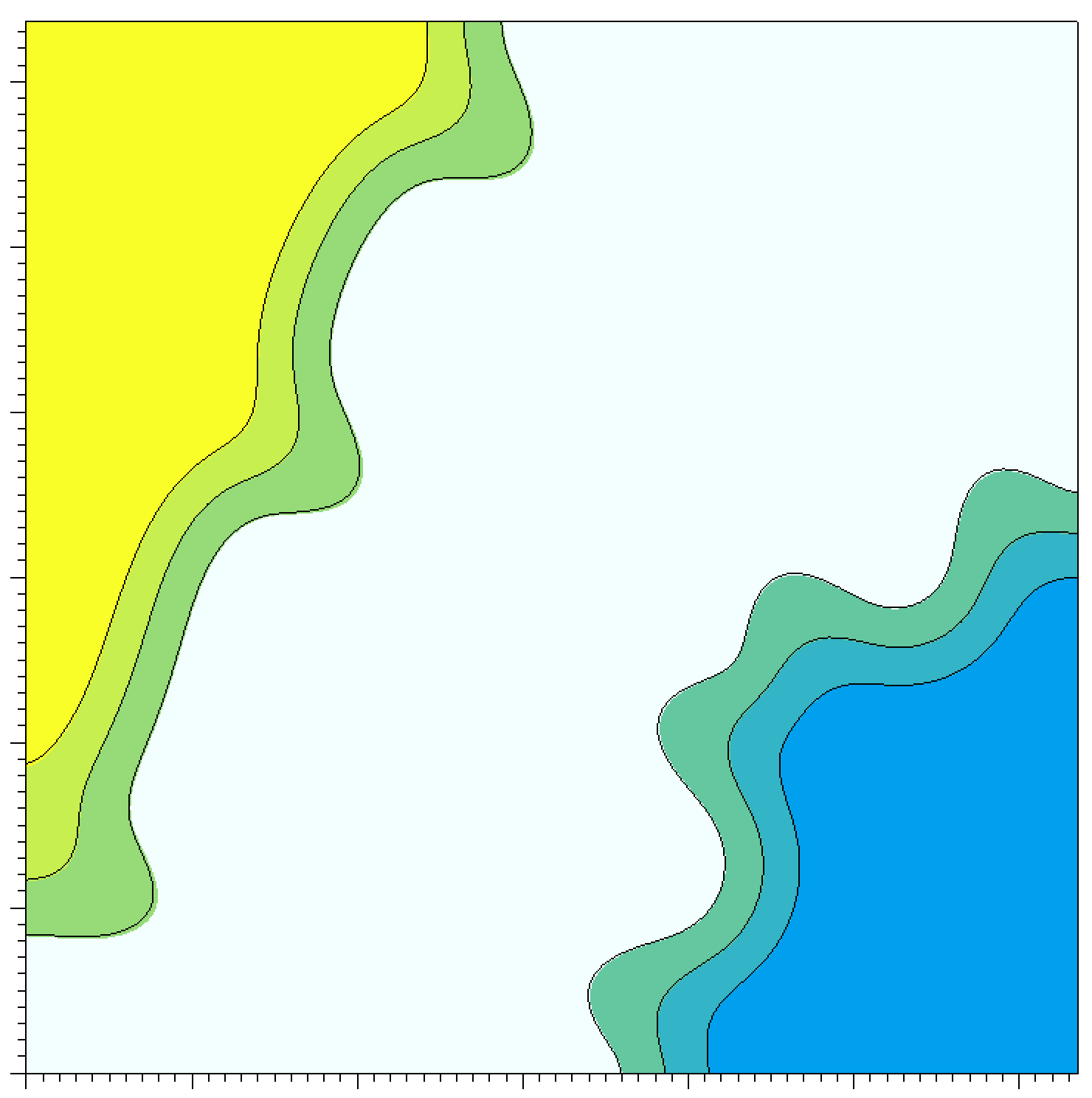}}
{\includegraphics[width=0.23\textwidth]{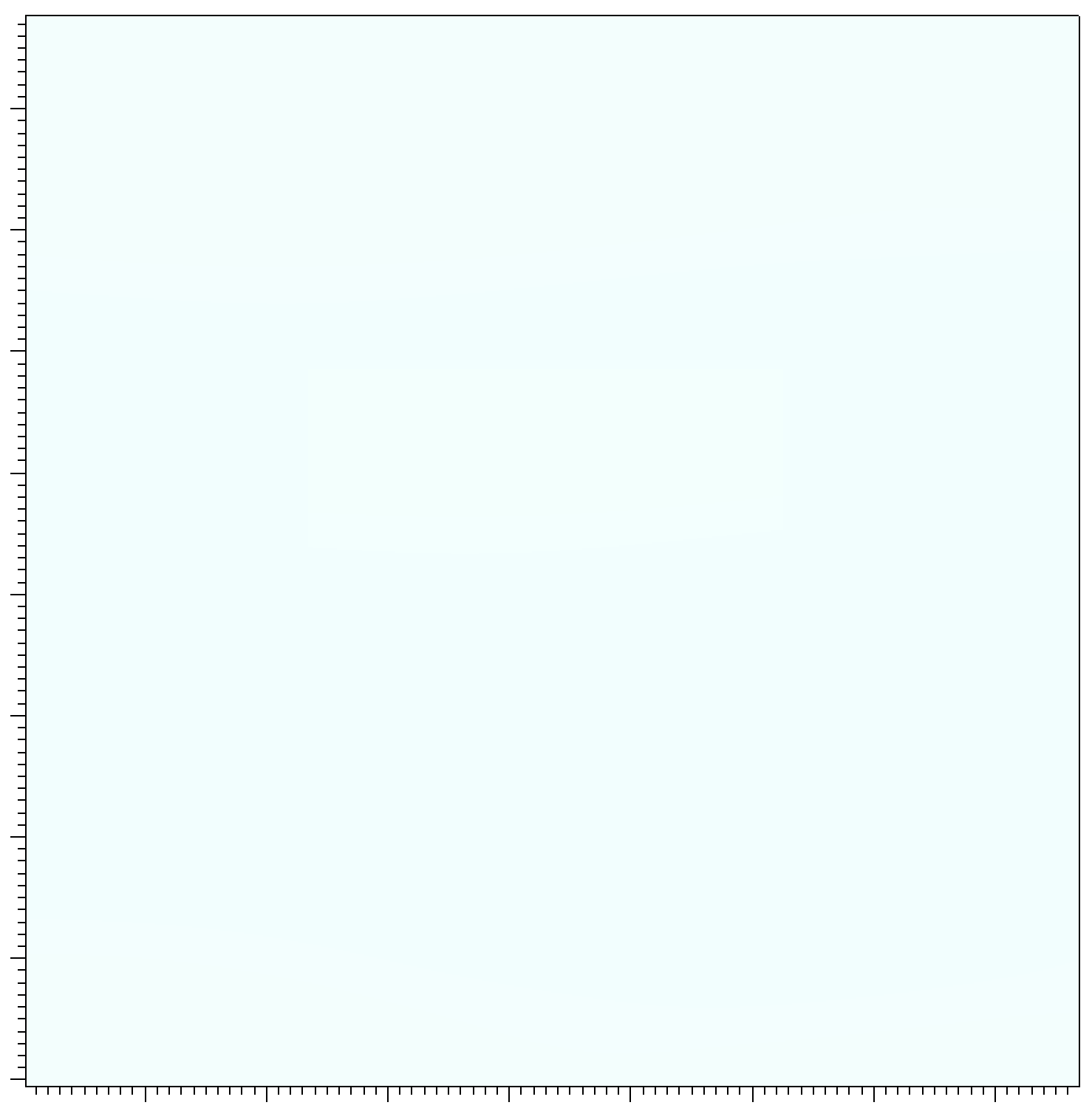}}
{\includegraphics[width=0.23\textwidth]{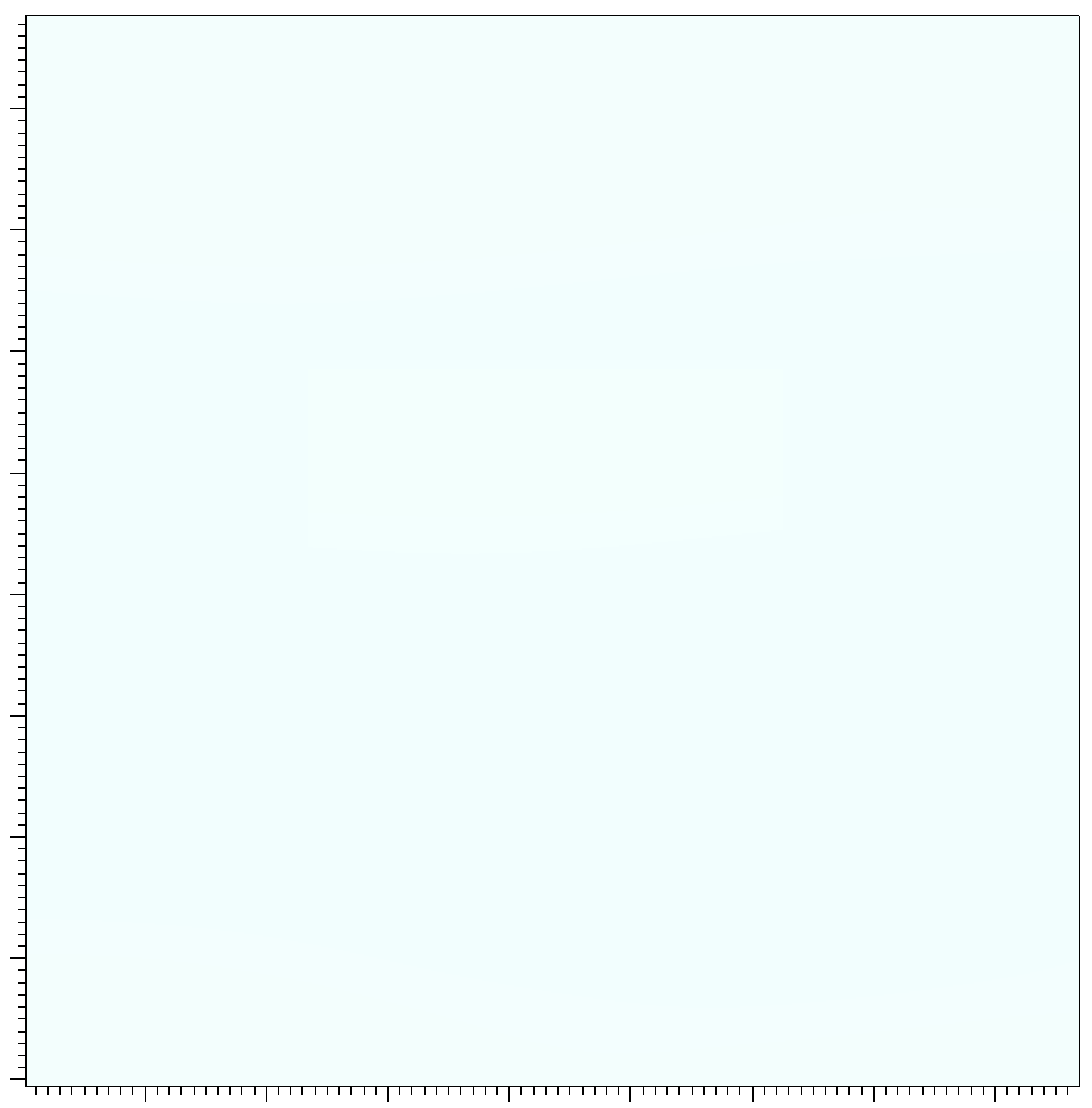}}
\end{center}
\caption{ Simulation a: homogeneous Neumann boundary conditions. The distribution of the concentration field $u$ (a1-a4) 
and order parameter $v$ (a5-a8) on the computational domain. \color{black} In Fig.~(a7) and (a8),  the order parameter $v$ in the total computational domain is very close to zero, which leads to empty contour plots.  \color{black}
 }
\label{ex2d:solu}
\end{figure}

\begin{figure}[!h]
\begin{center}
\scriptsize{(b1) $t = 0.4$}\qquad\qquad\qquad\qquad\scriptsize{(b2) $t = 1$}
\qquad\qquad\qquad\qquad\scriptsize{(b3) $t = 1.4$}
\qquad\qquad\qquad\qquad\scriptsize{(b4) $t = 10,000$}\\
{\includegraphics[width=0.03\textwidth]{draw/tab-u}}
{\includegraphics[width=0.23\textwidth]{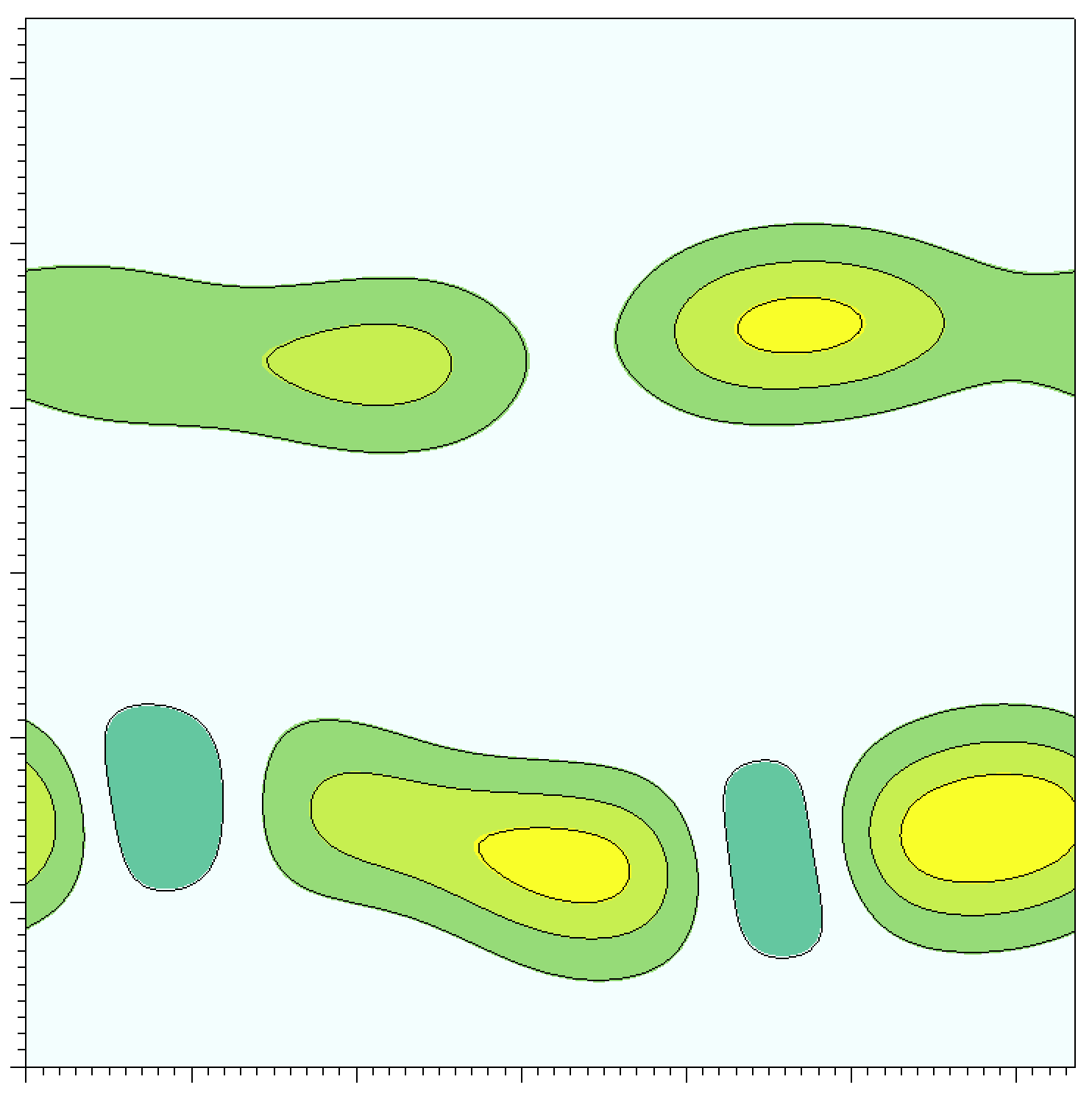}}
{\includegraphics[width=0.23\textwidth]{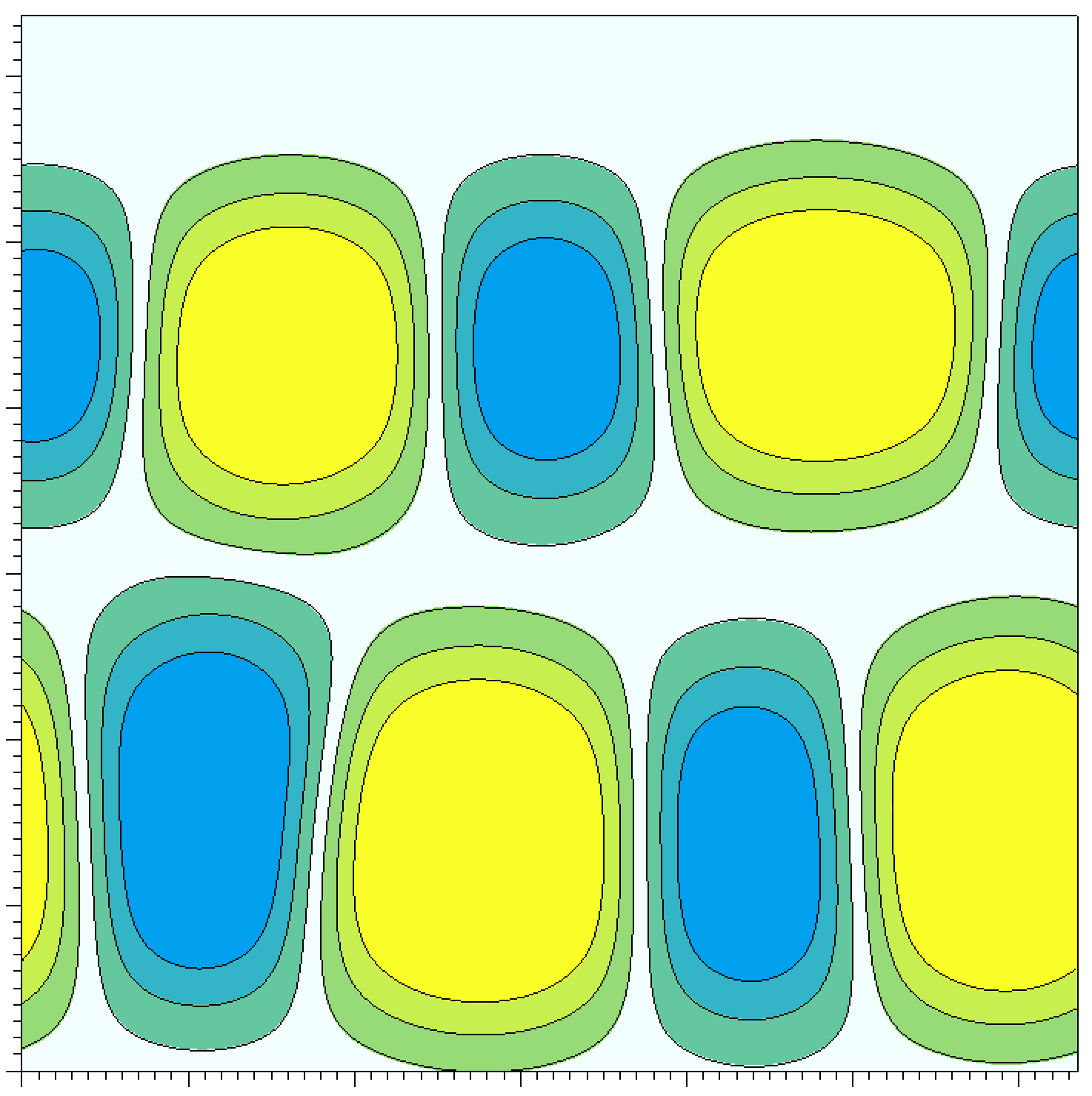}}
{\includegraphics[width=0.23\textwidth]{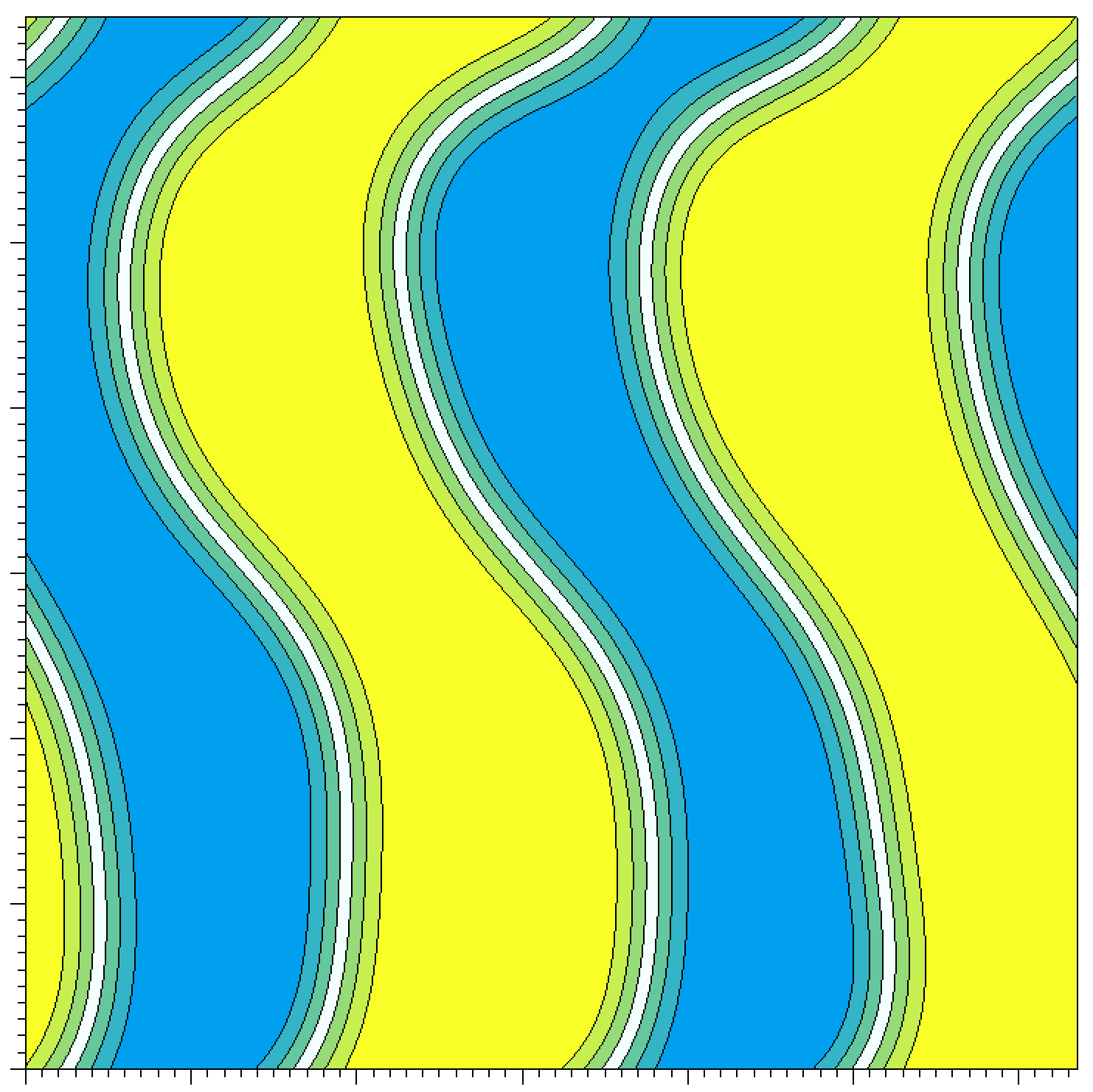}}
{\includegraphics[width=0.23\textwidth]{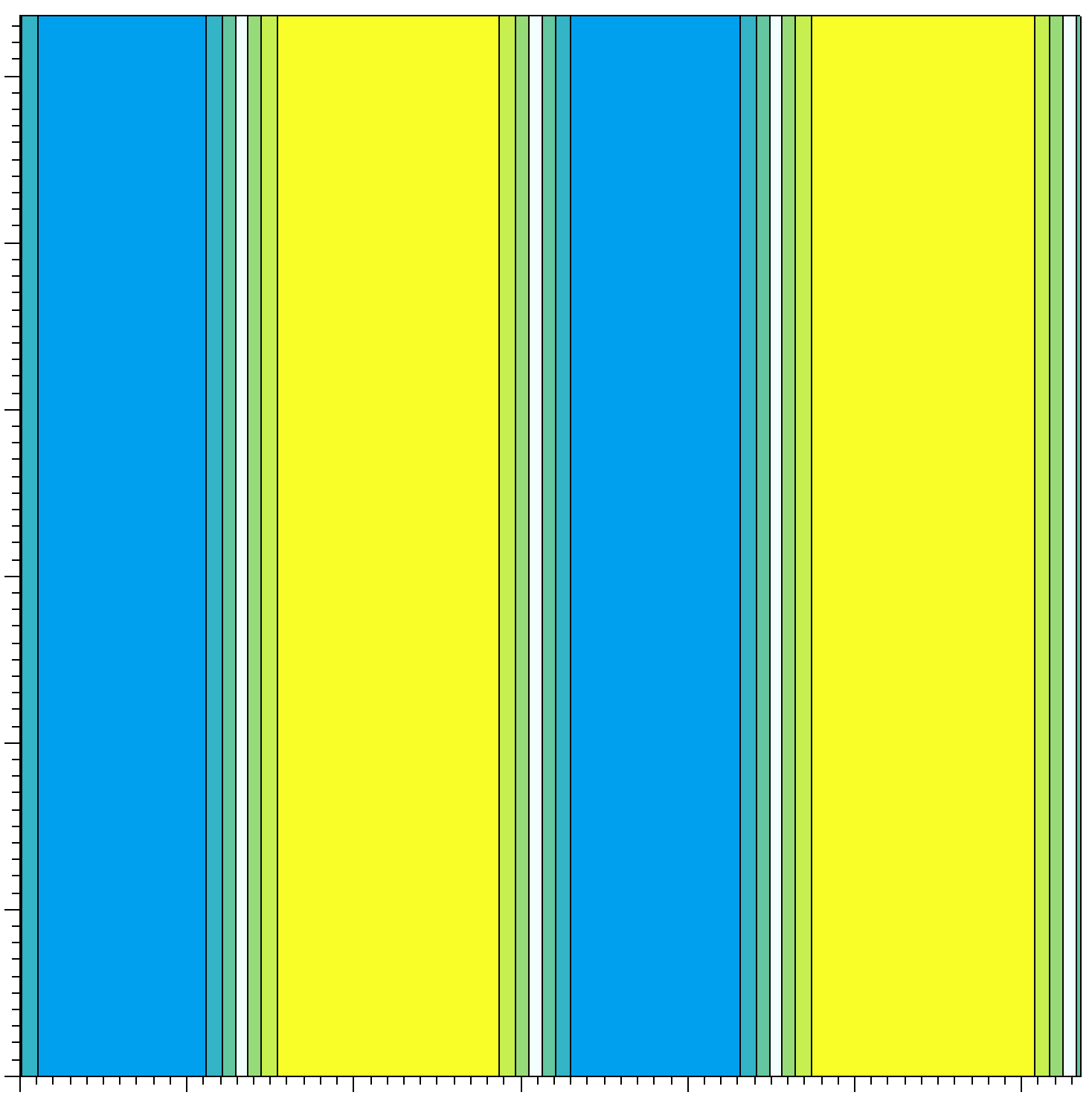}}
\\~~\\
\scriptsize{(b5) $t = 0.4$}\qquad\qquad\qquad\qquad\scriptsize{(b6) $t = 1$}
\qquad\qquad\qquad\qquad\scriptsize{(b7) $t = 1.4$}
\qquad\qquad\qquad\qquad\scriptsize{(b8) $t = 10,000$}\\
{\includegraphics[width=0.035\textwidth]{draw/tab-v}}
{\includegraphics[width=0.23\textwidth]{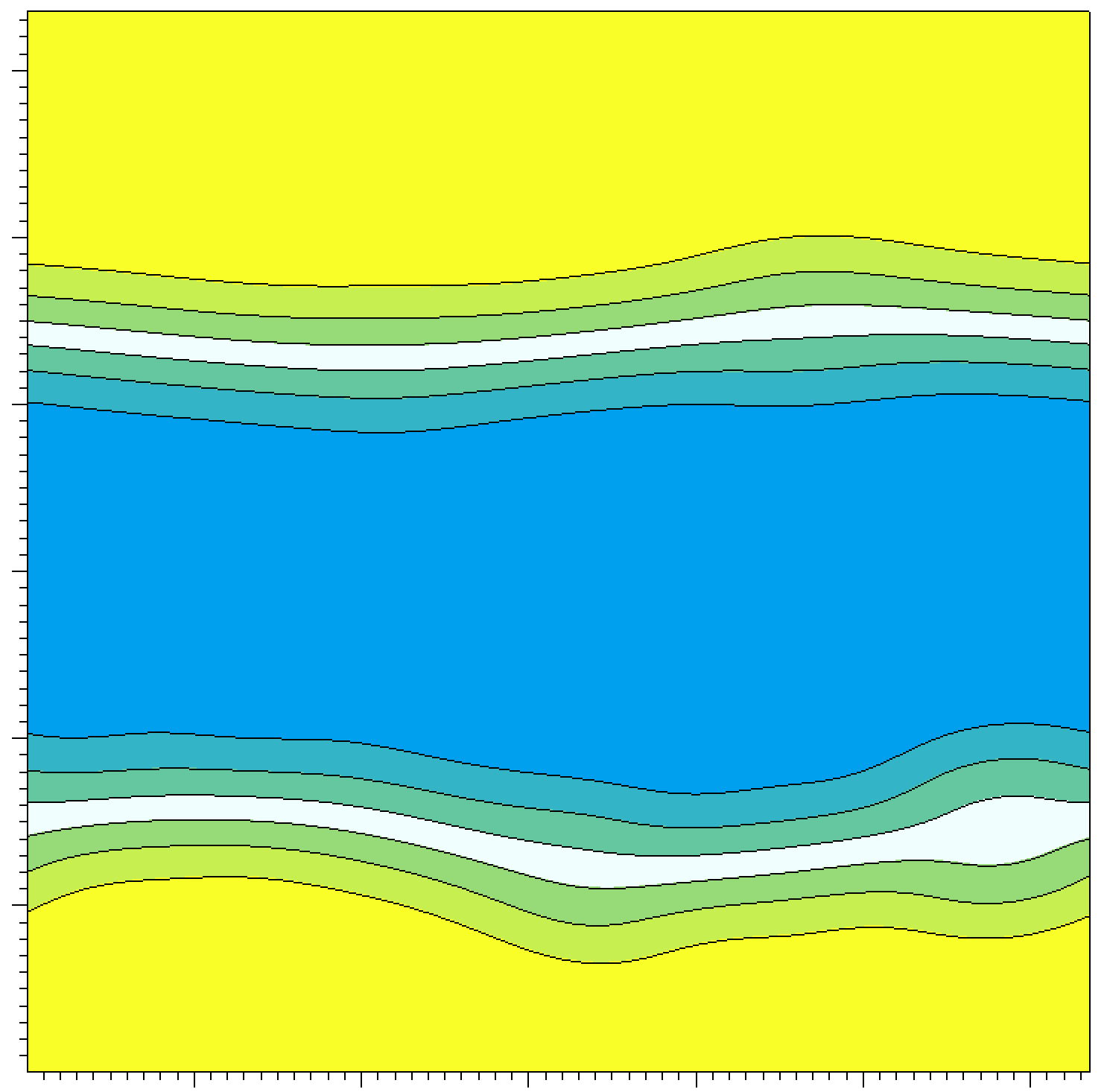}}
{\includegraphics[width=0.23\textwidth]{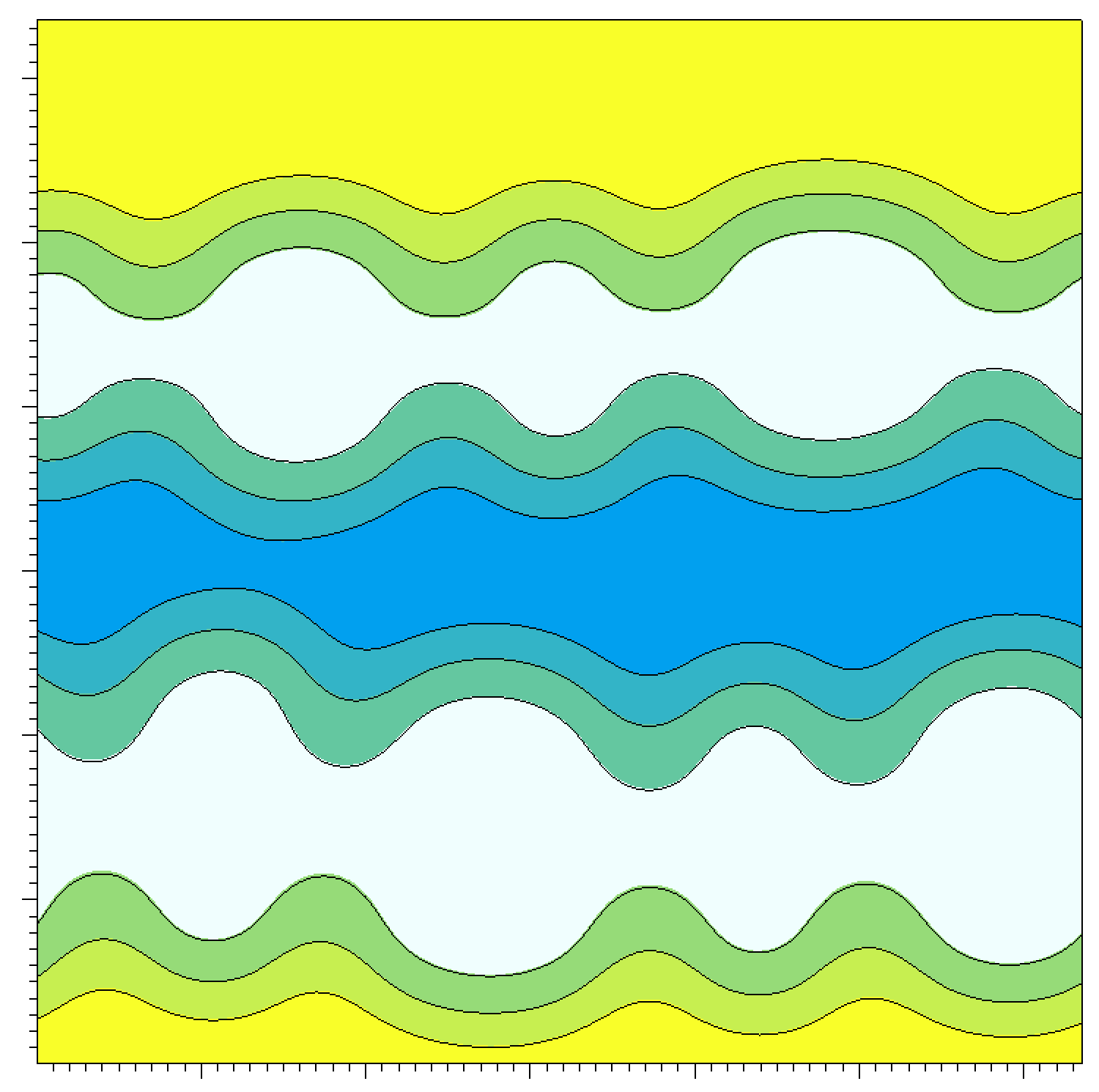}}
{\includegraphics[width=0.23\textwidth]{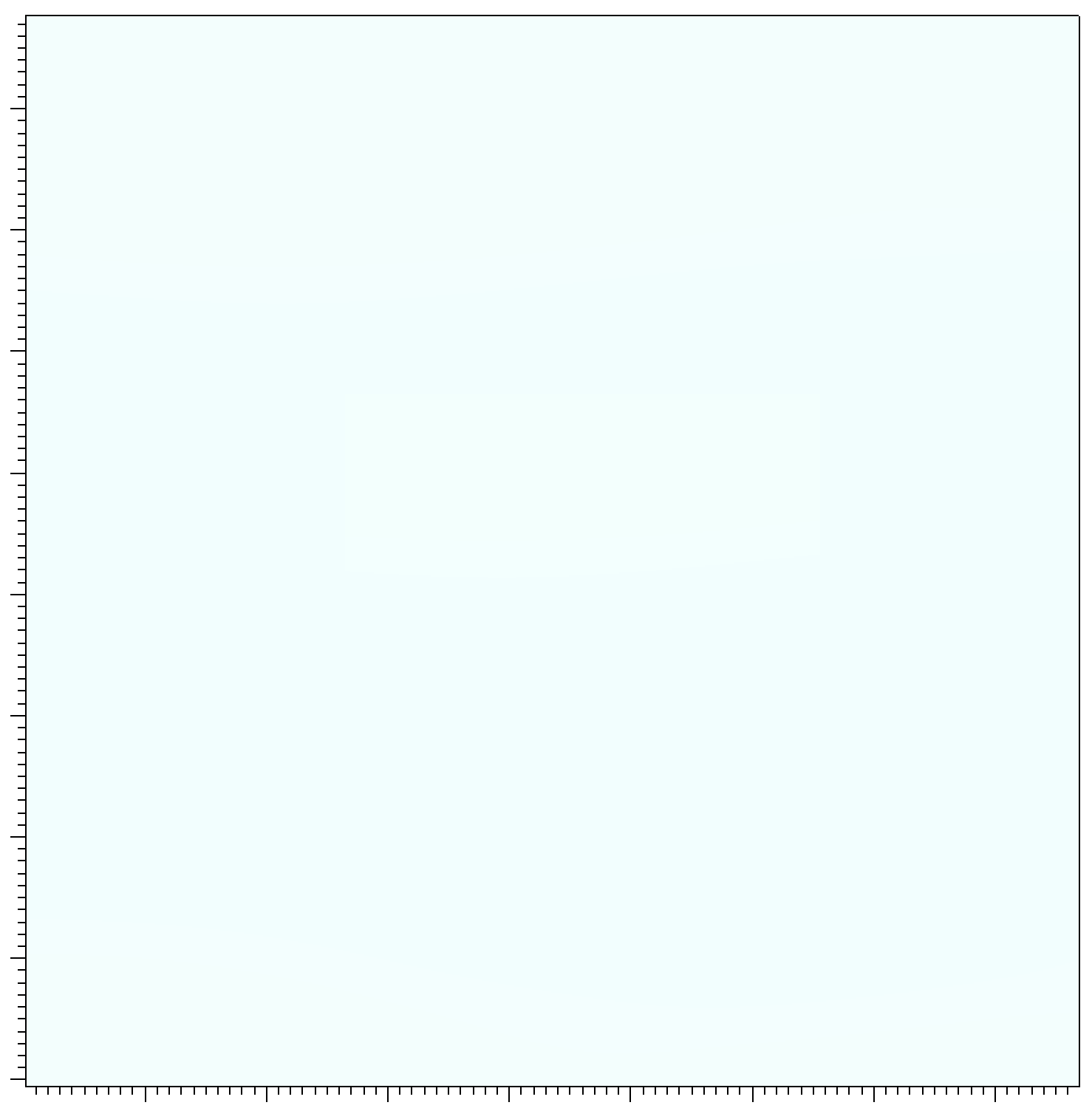}}
{\includegraphics[width=0.23\textwidth]{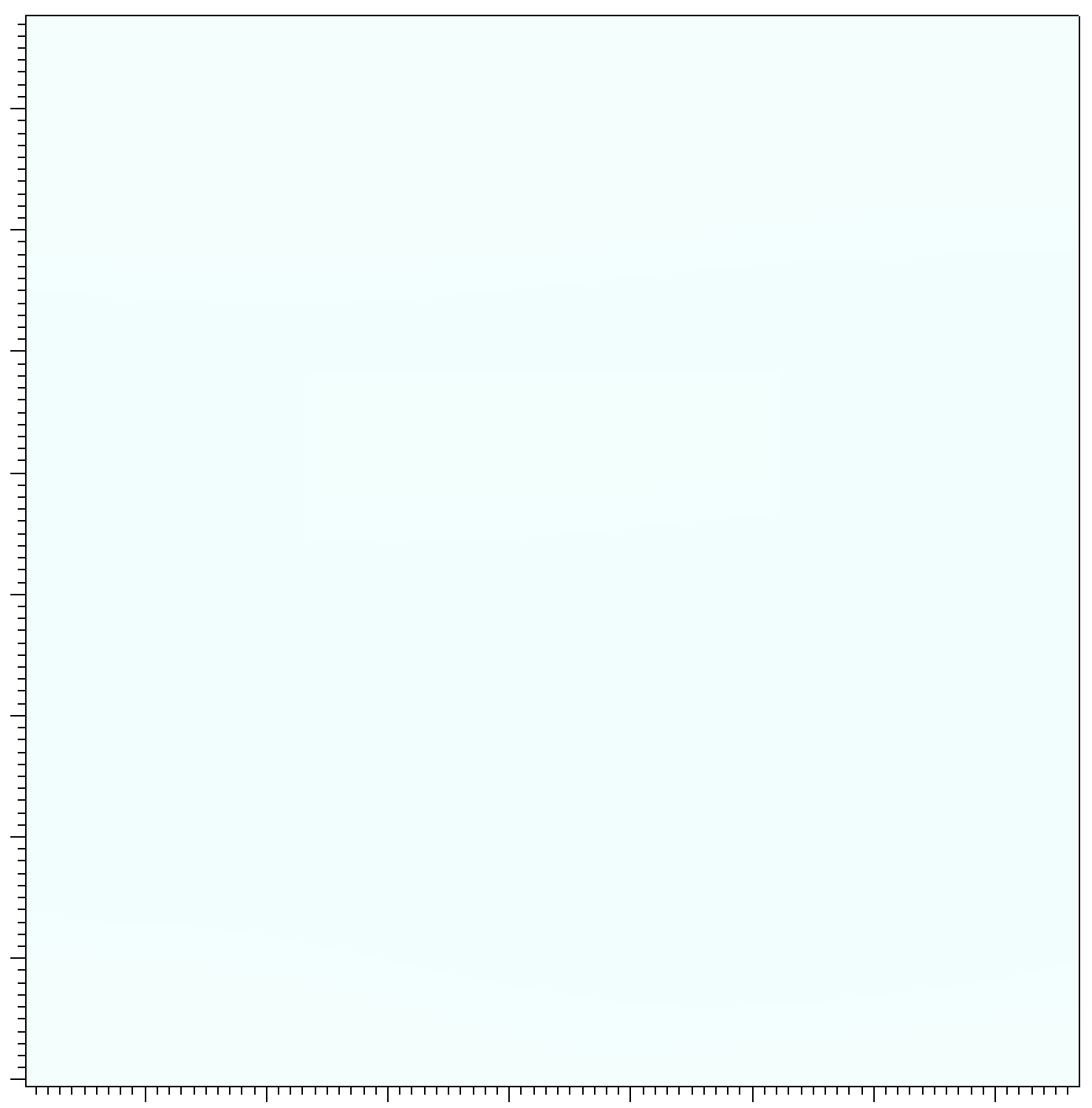}}
\end{center}
\caption{Simulation a: periodic boundary conditions. 
The distribution of the concentration field $u$ (b1-b4)  and  order parameter $v$ (b5-b8)  on the computational domain. 
\color{black} In Fig.~(b7) and (b8),  the order parameter $v$ in the total computational domain is very close to zero, which leads to empty contour plots.  \color{black} }
\label{ex2d:solv}
\end{figure}

To test the effect of the boundary condition, we run two simulations with: (a) homogeneous Neumann boundary conditions, and (b) periodic boundary conditions, respectively.
The contour plots of the concentration field $u$ and the order parameter $v$ are drawn in Fig.~\ref{ex2d:solu}  and Fig.~\ref{ex2d:solv}, respectively.
The numerical results show that both phase separation and order-disorder transitions 
occur at the early stage of the evolution, and an ordered steady state is finally reached after
the phase separation is completed.  As shown in Fig.~\ref{ex2d:solu},  simulation (a) 
only has one phase interface and  simulation (b) has four, with more intensity.
As shown in Fig.~\ref{ex2d:solv}, the order parameter $v$ first develops to the state approaching 
to the upper and lower bounds $\pm 1/2$, then quickly tends to zero as the concentration field $u$ 
coarsens to a steady state. 
Overall, the simulations results  agree well with published results
\cite{Xia1,yangcai12dd21}.

Furthermore, we plot the evolution of the total free energy and the history of the time step size  
in Fig.~\ref{ex2d:energy}. As seen from Fig.~\ref{ex2d:energy}(1), the total free energies 
of all simulations decrease monotonically as the solution evolves to the steady state. 
As compared to simulation (b), the steady state of simulation (a) has lower total free energy due to \color{black} fewer \color{black} phase interfaces.
From Fig.~\ref{ex2d:energy}(2), we observe 
the size of the time step for the both simulations initially keeps to be $\Delta t_{\mathrm{min}}$ due to the fast variation of the solutions, then increases due to phase separation and finally evolves to $\Delta t_{\mathrm{max}}$. By using the adaptive time stepping strategy, the time step is successfully adjusted by five orders of magnitude, which can substantially reduce the computational cost. 
As a comparison, we also run the two simulations
with a fixed time step size $\Delta t=10^{-4}$ and plot the evolutions of the total free-energy in Fig.~\ref{ex2d:energy}~(1).
From the figure, we observe that the evolutions of the total free-energy obtained from the adaptive time stepping strategy
and a fixed small time step size are almost the same,
which validates the accuracy of the adaptive time stepping strategy. The total compute times for the four simulations are listed
in Fig.~\ref{ex2d:energy} (1), which clearly shows that the adaptive time stepping strategy can save 99.9\% compute time as compared with a fixed small time step size.

\begin{figure}[!h]
\begin{center}
{\includegraphics[width=0.48\textwidth]{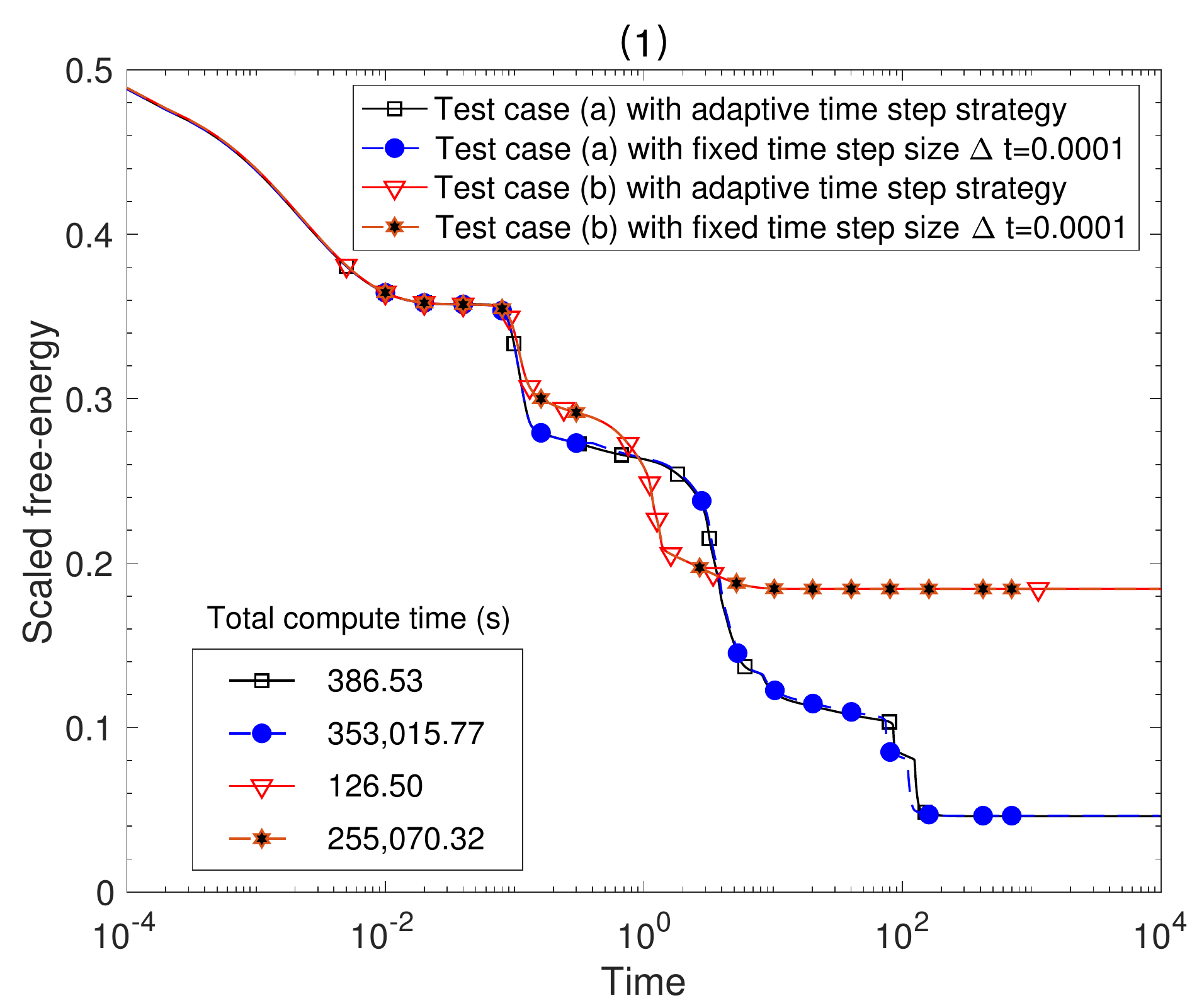}}
{\includegraphics[width=0.48\textwidth]{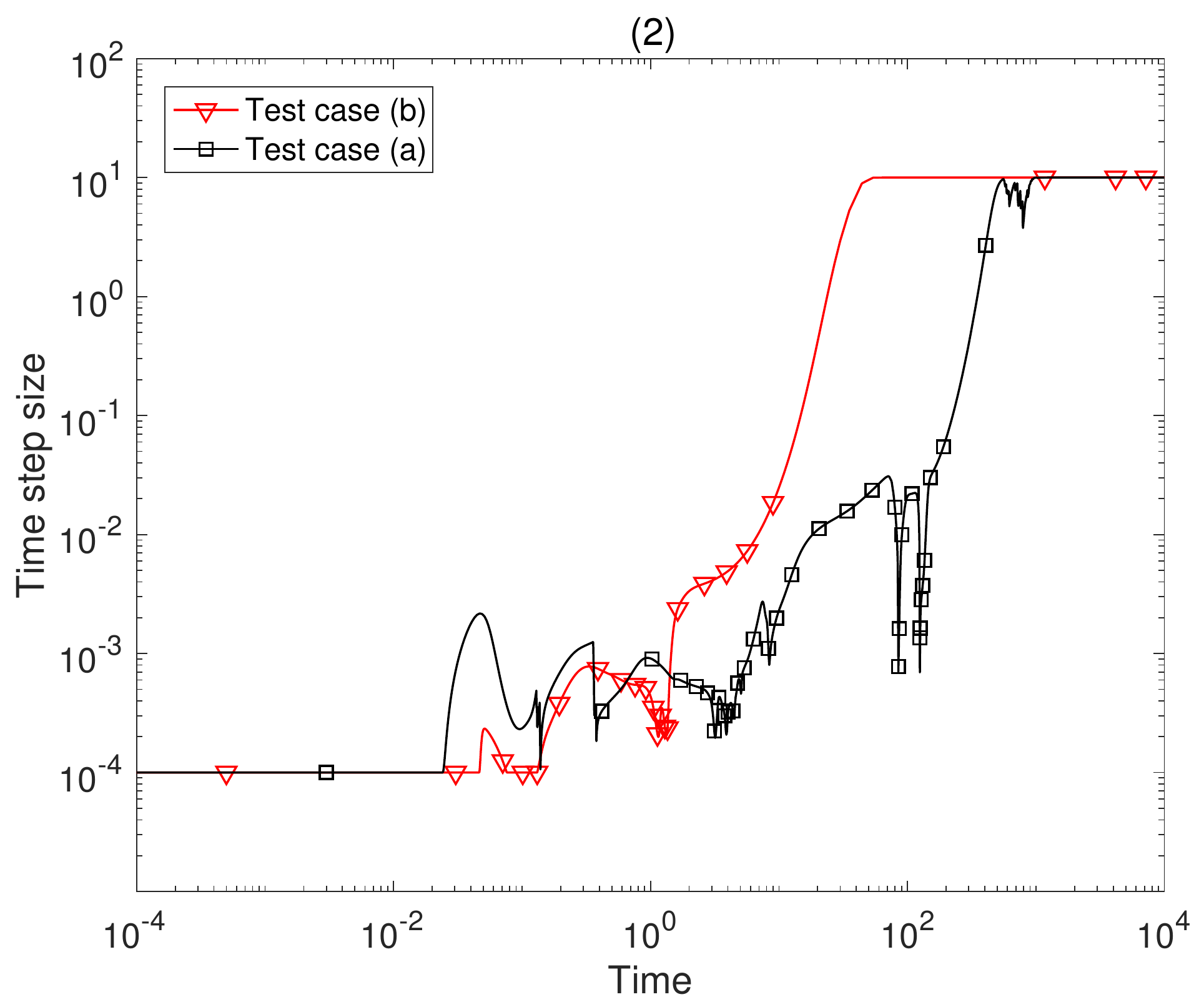}}
\end{center}
\caption{The evolution of the total free-energy (1) and the history of the time step size (2). All test cases are done by using 360 processor cores. }
\label{ex2d:energy}    
\end{figure}

\vspace{0.2cm}
\noindent B. Three dimensional tests
\vspace{0.2cm}

In this part, we consider a three dimensional problem with randomly initial data  $(u^{(0)}, v^{(0)})  = 
(0.55 + \delta_u, \delta_v)$. Here 􏰋$\delta_u$ and $\delta_v$ are 
uniform random distributions in $-0.05$ to $0.05$. The parameters are set to
$\alpha = 4$, $\beta = 2$, $\gamma=0.005$, $\theta = 0.1$, $\rho =0.001$, and $S=10$. The computational domain 
 $\Omega = [0,1]^3$ is covered by a $128\times 128\times 128$ uniform mesh. The time step size is  initially set to
$\Delta t_1 = 10^{-4}$ and  adaptively controlled by the proposed adaptive time stepping strategy with
 $\Delta t_\mathrm{min} = 10^{-4}$, and $\Delta t_\mathrm{max} = 2$.
We run the test with the homogeneous Neumann boundary conditions. 
\begin{figure}[!h]
\begin{center}
\scriptsize{(a1) $t = 1$}\qquad\qquad\qquad\qquad\scriptsize{(a2) $t = 2$}
\qquad\qquad\qquad\qquad\scriptsize{(a3) $t = 10$}\qquad\qquad\qquad\qquad\scriptsize{(a4) $t = 3,000$}\\
{\includegraphics[width=0.038\textwidth]{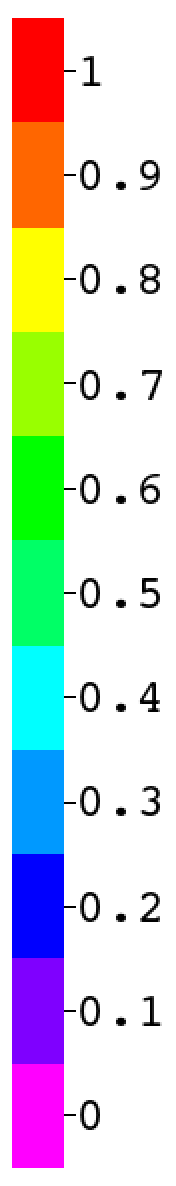}}
{\includegraphics[width=0.23\textwidth]{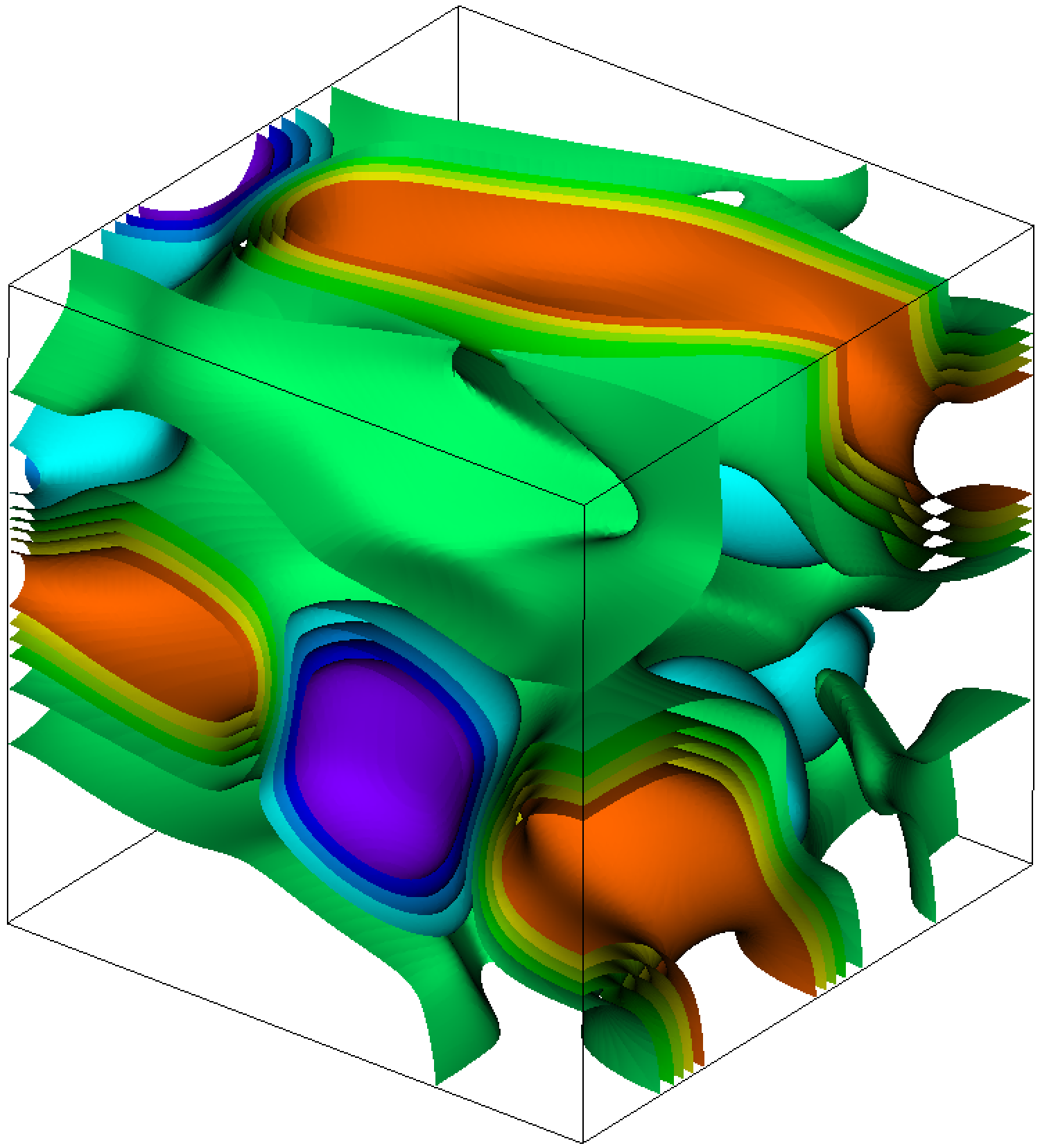}}
{\includegraphics[width=0.23\textwidth]{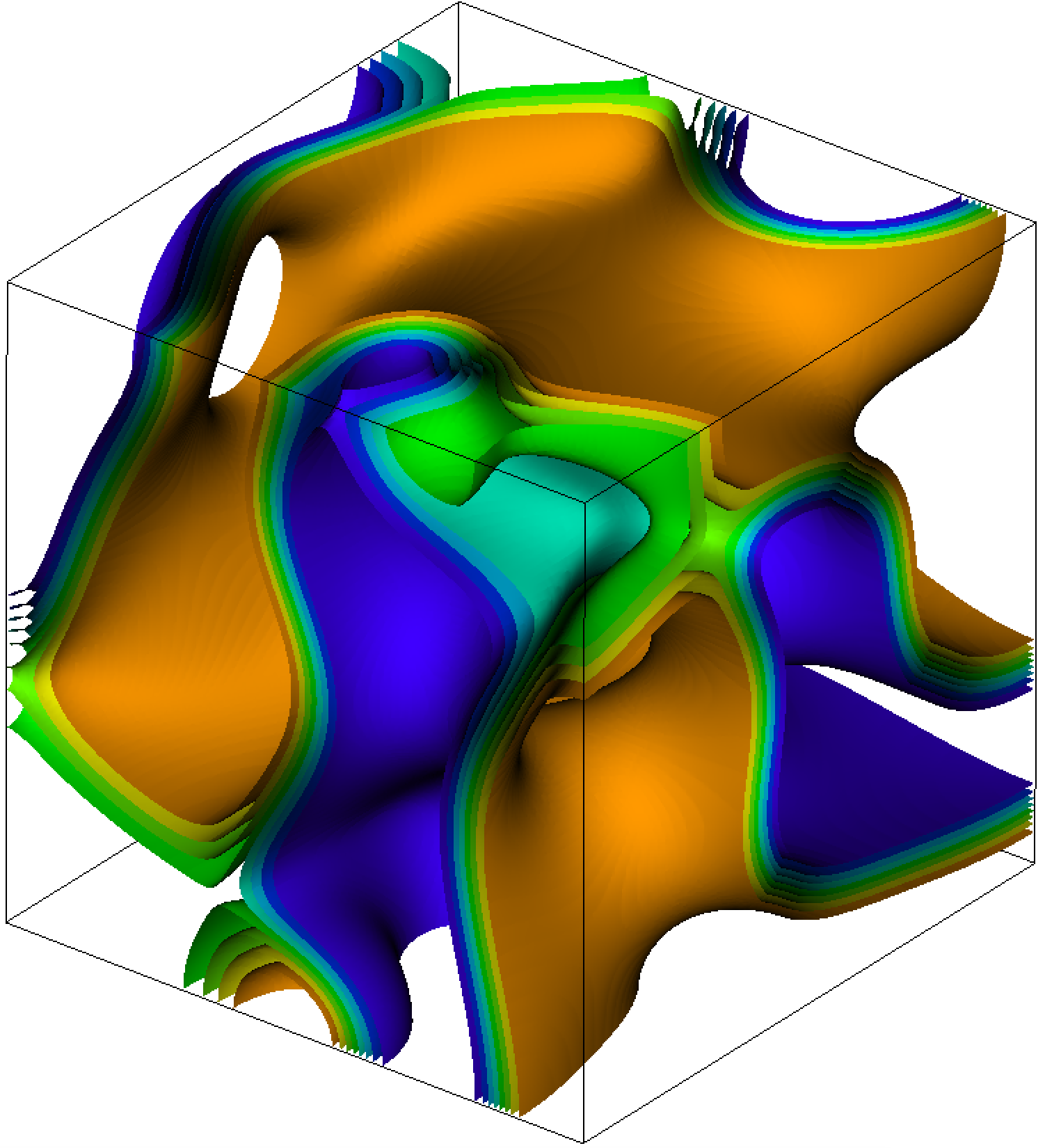}}
{\includegraphics[width=0.23\textwidth]{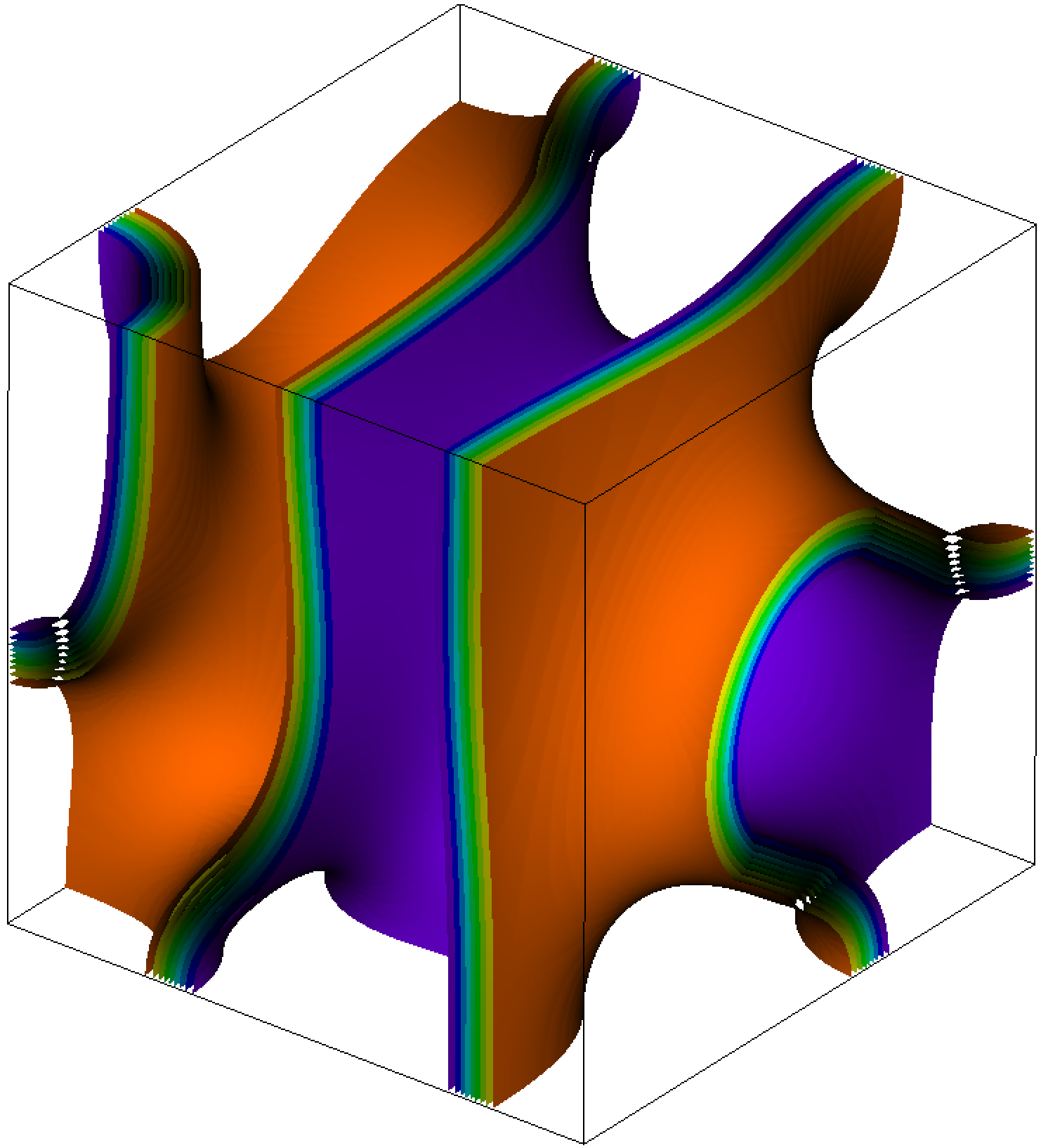}}
{\includegraphics[width=0.23\textwidth]{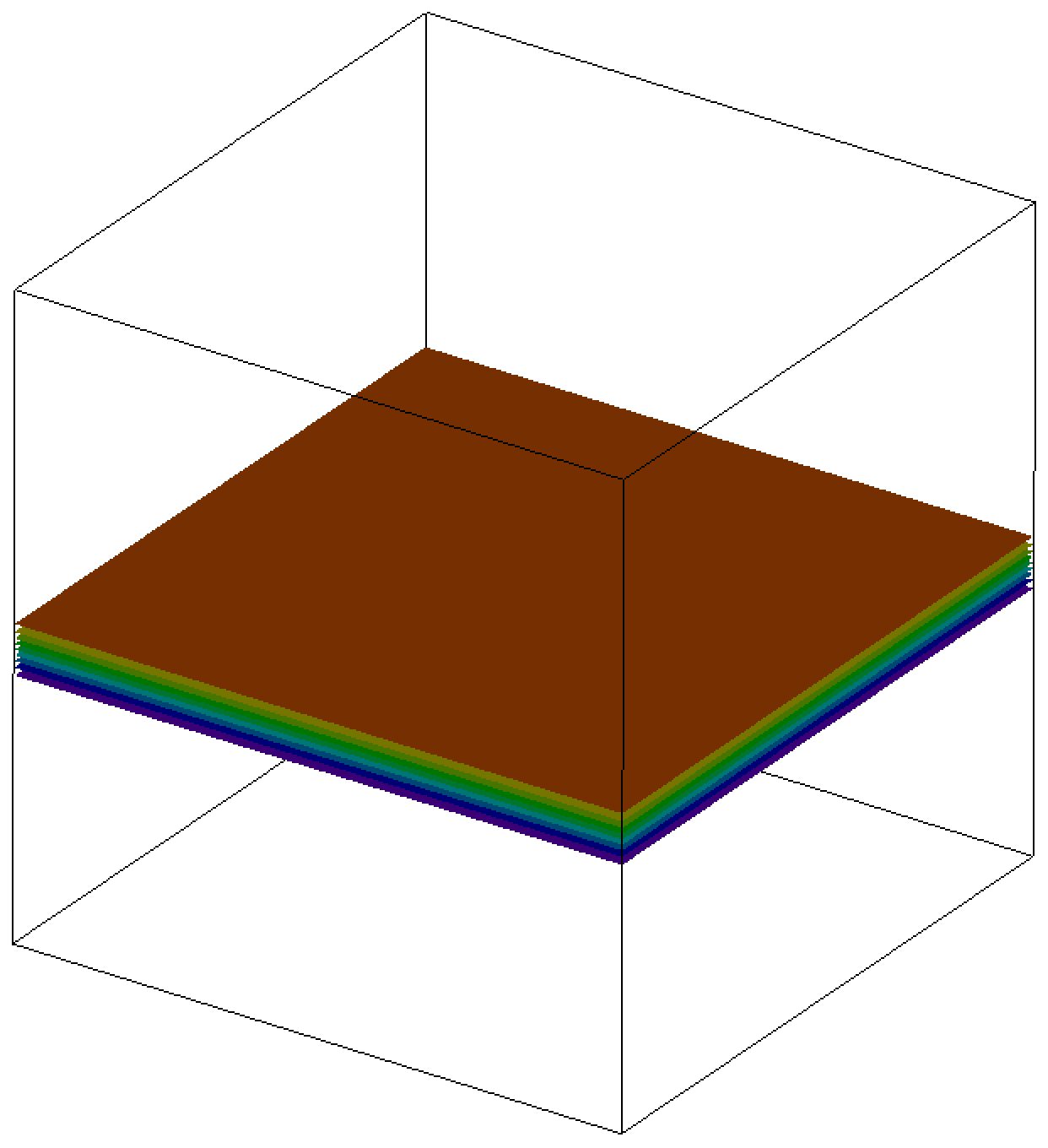}}
\\~~\\
\scriptsize{(b1) $t = 1$}\qquad\qquad\qquad\qquad\scriptsize{(b2) $t = 2$}\qquad\qquad\qquad\qquad\scriptsize{(b3) $t = 10$}\qquad\qquad\qquad\qquad\scriptsize{(b4) $t = 3,000$}\\
{\includegraphics[width=0.042\textwidth]{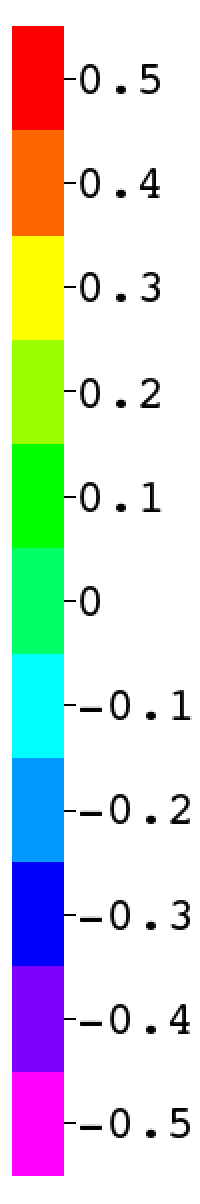}}
{\includegraphics[width=0.23\textwidth]{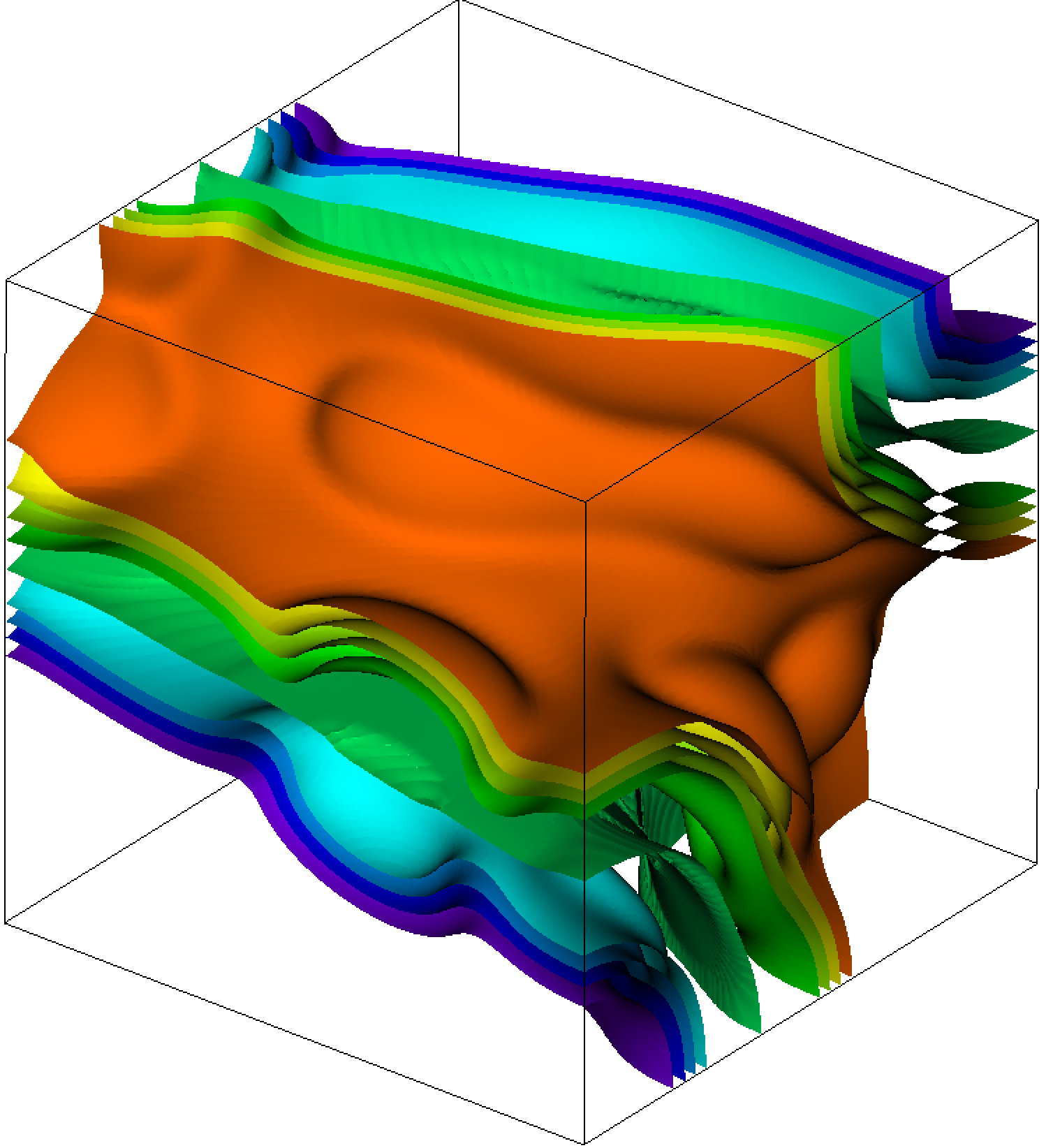}}
{\includegraphics[width=0.23\textwidth]{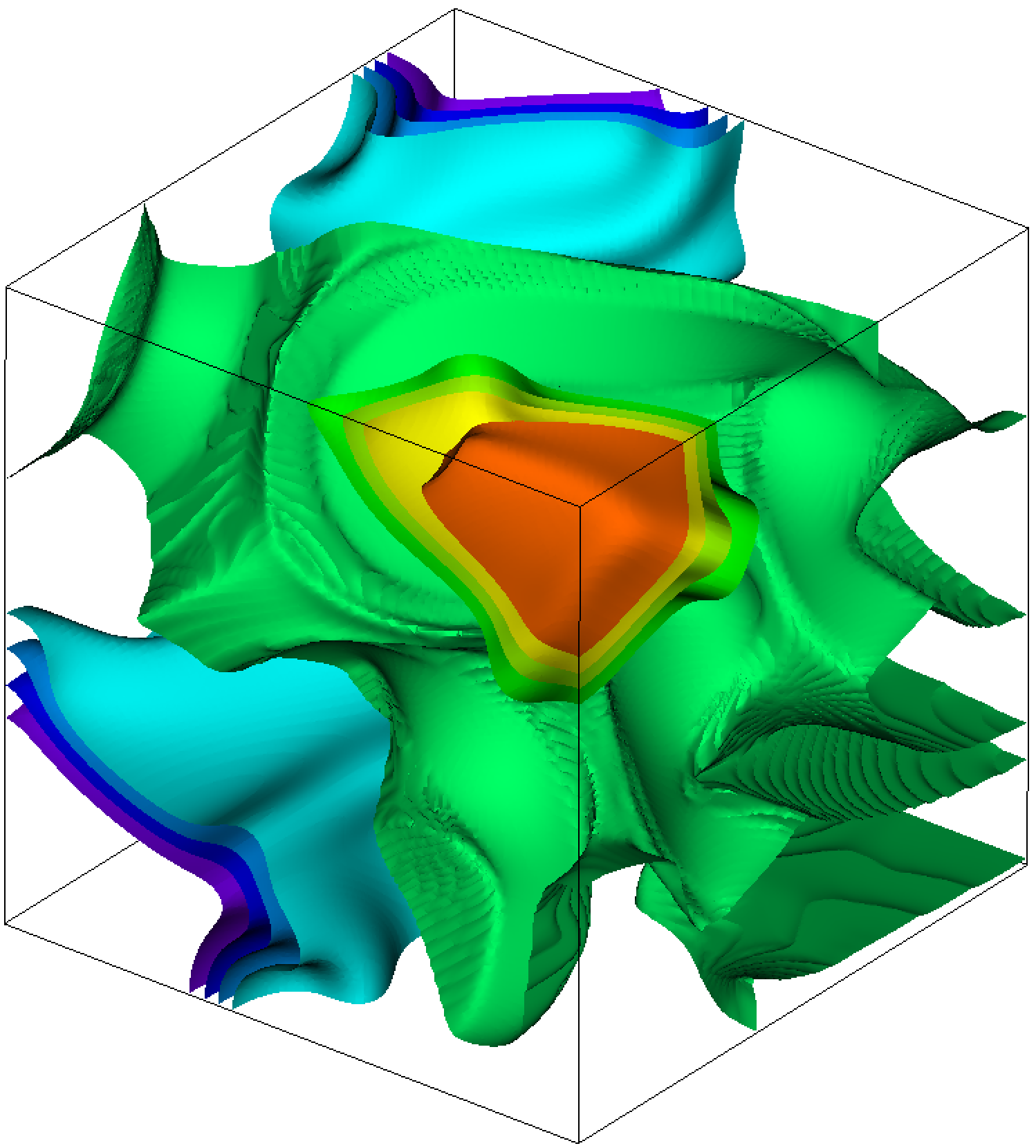}}
{\includegraphics[width=0.23\textwidth]{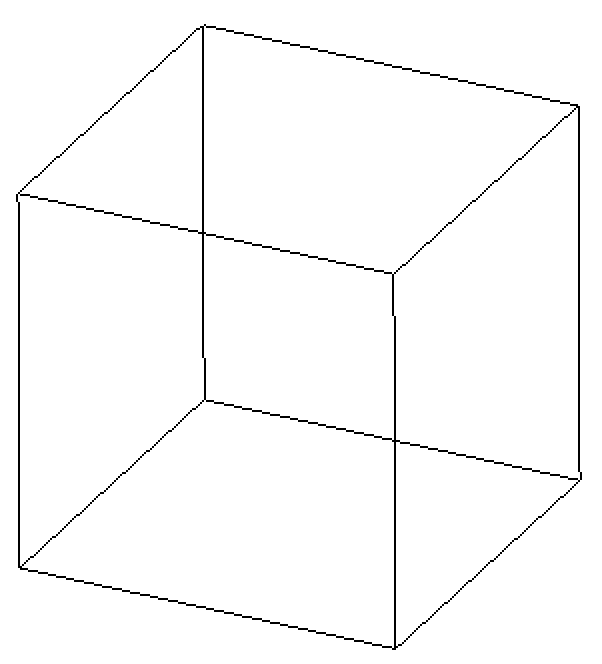}}
{\includegraphics[width=0.23\textwidth]{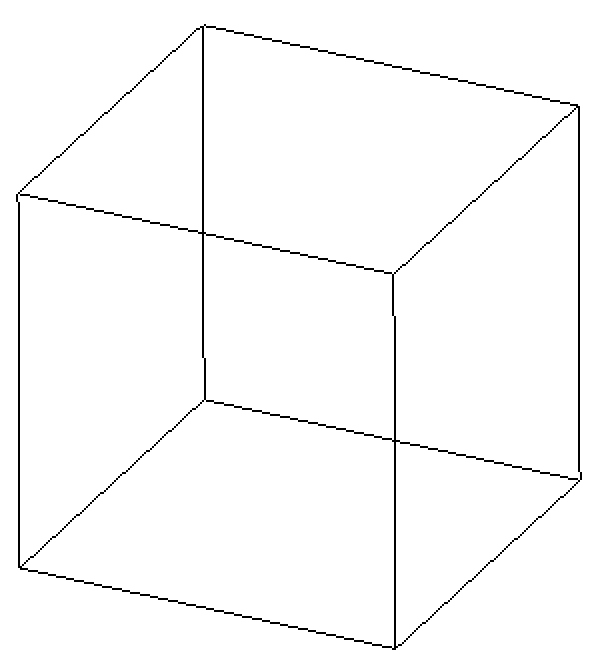}}
\end{center}
\caption{The distribution of the concentration field $u$ (a1-a4) and the order parameter $v$
(b1-b4). \color{black} In Fig.~(b3) and (b4),  the order parameter $v$ in the total computational domain is very close to zero, which leads to empty iso-surface plots.  \color{black}
}
\label{ex3d:isosurface}
\end{figure}

Fig.~\ref{ex3d:isosurface} displays the isosurface plots of the concentration field $u$ 
and the order parameter $v$ at $t=1, 2, 10,$ and 3,000, respectively.  
As seen from Fig.~\ref{ex3d:isosurface}, the phase separation and order-disorder transitions occur at the beginning, and the 
order parameter $v$ quickly tends to zero as the concentration field $u$ coarsens to a steady state, 
which is similar to the two dimensional case.
The evolution of the total free energy and the 
history of the time step size 
are shown in Fig.~\ref{ex3d:energy}. From Fig.~\ref{ex3d:energy}, we observe that the
total free energy decreases monotonically as the solution evolves to the steady state
and the time step size is successfully adjusted from 
$\Delta t_{\min}$ to $\Delta t_{\max}$ by four orders of magnitude. 
In this simulation, the parameter $\eta$ in the adaptive time stepping strategy is initially set as 100 and finally adjusts to $2^{36}\times100$, 
which is ten orders of magnitude larger.
To show the efficiency of the adaptive adjustment of $\eta$, we rerun the simulation in the time interval $[680,\,800]$ by using the adaptive time stepping 
strategy with fixed \color{black}$\eta=$3,200 \color{black} as a comparison. The corresponding compute time and the total number of divergent NKS solvers are 
reported in Table~\ref{tab:comparison}, which shows that one can  save about 50\% compute time by using an automatically adjusted $\eta$.

\begin{figure}[!h]
\begin{center}
{\includegraphics[width=0.48\textwidth]{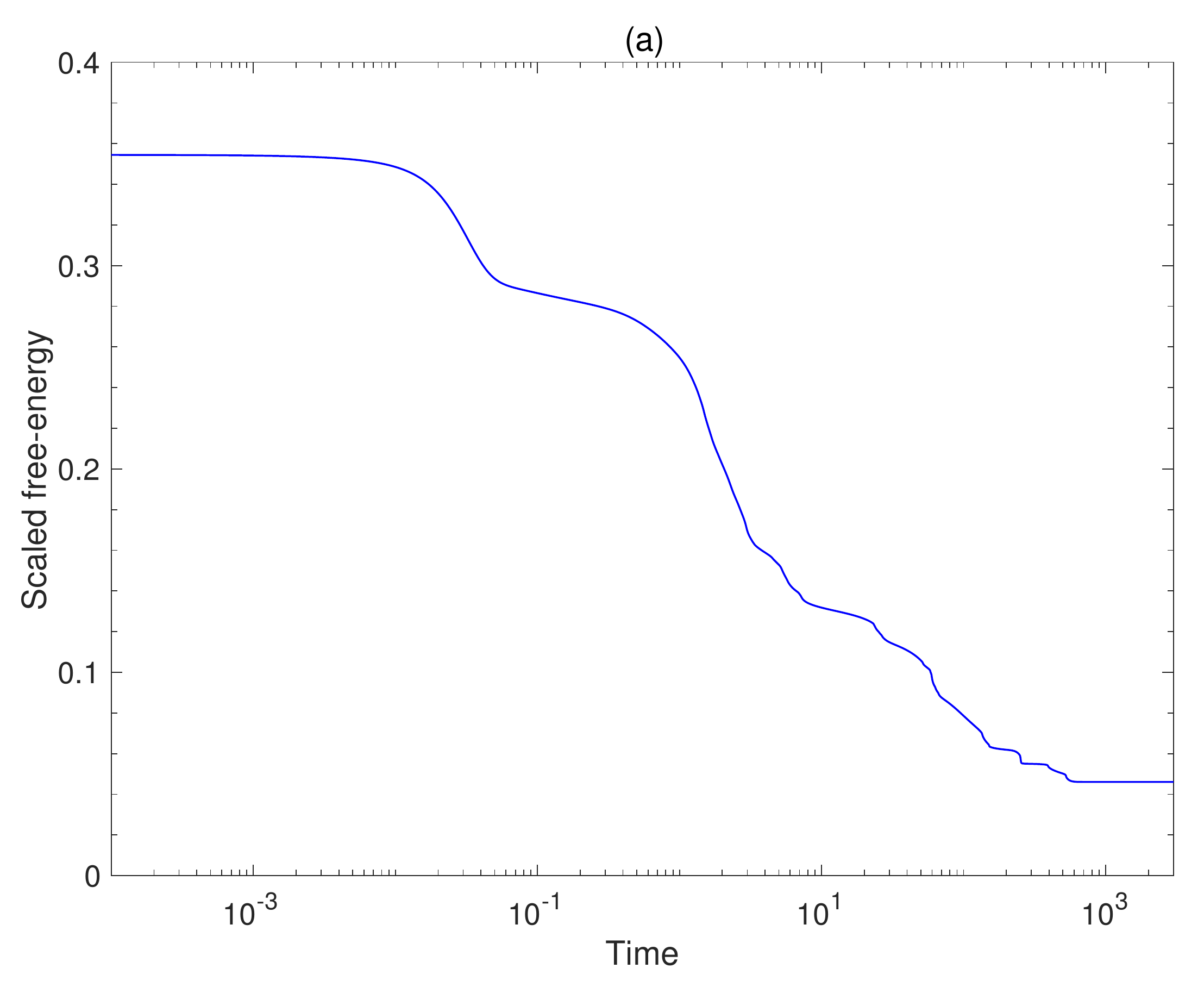}}
{\includegraphics[width=0.48\textwidth]{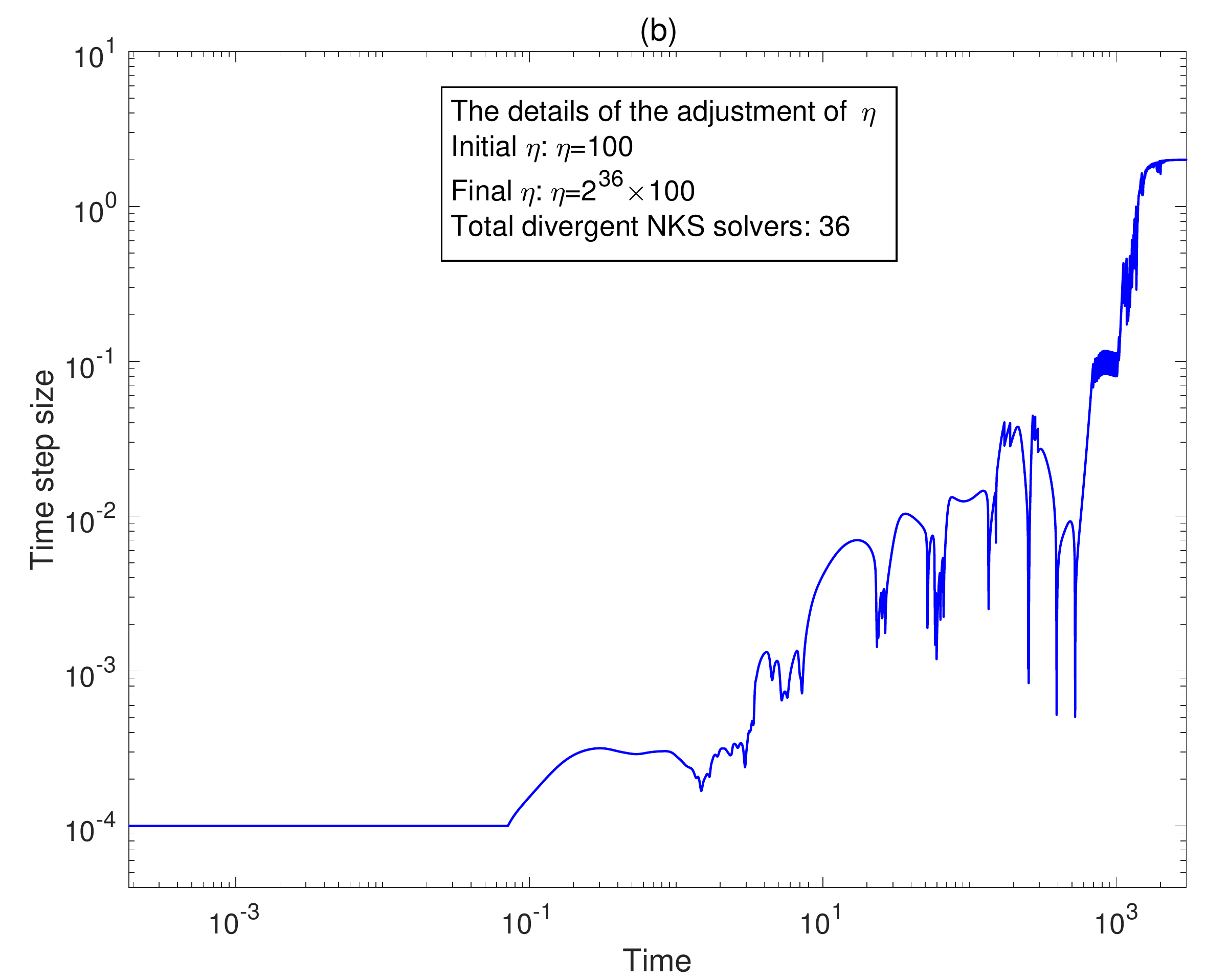}}
\end{center}
\caption{The total free-energy (a) and the history of the time step size (b) in the three dimensional test case. }
\label{ex3d:energy}    
\end{figure}

\begin{table}[!htb]
\caption{
The comparison between the adaptive time stepping strategy with an automatically adjusted $\eta$ and the adaptive time stepping strategy with a fixed $\eta=3,200$.
All test cases are done by using 720 processor cores.
}
\label{tab:comparison} 
\centering
\begin{tabular}{|c|cc|}
\hline
& Total number of divergent NKS solvers& Total compute time (s)\\
\hline
\hline
Fixed $\eta=3200$& 3383& 87,812.6\\

Adjusted $\eta$& 6& 46,422.4\\
\hline
\end{tabular}
\end{table}

\subsection{Performance tuning}
In this subsection, we focus on the parallel performance of the proposed algorithm. 
We run the  two  and three dimensional tests on a $128\times 128$  mesh and a $128 \times 128 \times 128$ mesh with 24 and 512 processors, 
respectively.
To accurately analyze the parallel performance,  we only run  the first $10$ time steps
with a fixed 
time step size $\Delta t =10^{-4}$. 
Homogeneous 
Neumann boundary conditions are applied in all tests.

First, we examine the influence of different subdomain solvers by limiting the 
test to the classical AS preconditioner and fixing the overlapping size to $\delta=1$. 
The ILU factorizations with $0$, $1$, and $2$ levels of fill-in 
and LU factorization are considered. 
The number of Newton iterations and the averaged number of GMRES iterations together with the total compute times are provided in Table~\ref{tab:subdomain}.
From Table~\ref{tab:subdomain}, we can see that 
the number of Newton iterations is insensitive to the subdomain solver.
In addition, the number of GMRES iterations can be reduced by increasing the fill-in level, 
but the total compute time keeps growing due to the increased cost of the subdomain 
solver. We find that the optimal choice in terms of the total compute time is the ILU(0) 
subdomain solver.
In the Newton method, the Jacobian matrices of the different Newton iterations  
have similar structures, so it is possible to save the compute time by 
only performing the subdomain matrix factorizations at the first step of the Newton iteration and reusing the factorized matrices
within the across different Newton iterations within the same time step. 
The results on applying the reuse strategy to the ILU(0) subdomain solver  are listed in the last column of Table~\ref{tab:subdomain}, which 
clearly shows that the reuse strategy can save nearly 20-30\% of the compute time.

\begin{table}[!htb]
\caption{
Performance of the NKS algorithm with different subdomain solvers.
}
\label{tab:subdomain} 
\centering
\begin{tabular}{|c|c|ccccc|}
\hline
\multicolumn{2}{|c|}{Subdomain solver} &ILU(0) &ILU(1) &ILU(2) &LU &ILU(0)-reuse \\
\hline\hline
\multirow{3}{*}{2D test} & Total Newton        &30            &30               &30              &30               &30  \\
                                    & GMRES/Newton  &10.56        &8.56            &8.53           &7.93            &11.7 \\
                                    &Total Time (s)        &3.17          &3.38            &3.81           &8.38            &2.57 \\
                                                                    
\hline
\multirow{3}{*}{3D test} & Total Newton          &20                &20                     &20                    &20                     &20  \\
                                    & GMRES/Newton      &21.45           &18                     &17                    &16.45                &21.6 \\
                                    & Total Time (s)           &74.88          &388.74              &1,442.51            &5,161.80           &53.90 \\
\hline
\end{tabular}
\end{table}

We then investigate the performance of the NKS solver by changing the type of the AS 
preconditioner and the overlapping factor $\delta$. The number of processor cores and 
the mesh size for the two test cases 
are the same with the previous simulations. Based on the previous report, we take 
the ILU(0)-reuse as the subdomain solver throughout the test cases. 
The classical-AS, the left-RAS, and the right-RAS preconditioners with overlapping size $\delta=0,\, 1,\, 2$ 
 are considered. 
The number of Newton iterations, the averaged number of GMRES iterations, and the total compute 
times are listed in Table \ref{tab:overlap}. From Table \ref{tab:overlap}, we can again see that the 
number of Newton iterations is insensitive to the type of the AS preconditioner and the overlapping 
size. We also conclude that the left-RAS and right-RAS preconditioners are superior to 
the classical-AS preconditioner for the two test cases. Moreover, the left-RAS and right-RAS 
preconditioners have  almost the same performance in terms of both the averaged number of GMRES iterations and the compute times. And the minimal compute time 
is achieved when $\delta=1$ for both the two and  three dimensional test cases. 

\begin{table*}[!hbt]
\caption{\upshape
Performance of NKS solver with respect to the types of preconditioner and overlapping sizes. }
\label{tab:overlap} 
\centering
\begin{tabular}{|c|c|ccc|cc|cc|}
\hline
\multicolumn{2}{|c|}{Type of preconditioner}
&\multicolumn{3}{|c|}{classical-AS}
&\multicolumn{2}{|c|}{left-RAS}
&\multicolumn{2}{|c|}{right-RAS}\\
\hline
\multicolumn{2}{|c|}{$\delta$}  &$0$ &$1$ &$2$ &$1$ &$2$   &$1$ &$2$   \\
\hline\hline
\multirow{3}{*}{2D test} & Total Newton        &30              &30         &30           &30             &30              &30            &30       \\
                                    & GMRES/Newton  &21.67          &11.7       &13.17      &6.8            &6.2             &6.83        &6.27     \\
                                    &Total Time (s)        &2.90            &2.57       &2.86        &2.15         &2.21            &2.16         &2.22   \\
\hline
\multirow{3}{*}{3D test} & Total Newton          &20                 &20                 &20              &20             &20               &20         &20      \\
                                    & GMRES/Newton      &43.15            &21.55            &25.1           &11             &10.5            &11         &10.5     \\
                                    & Total Time (s)           &41.88          &43.14             &67.23          &34.45         &50.98        &34.54   &51.06  \\
\hline
\end{tabular}
\end{table*}

\subsection{Weak and strong scaling tests}
In this subsection, we study the weak scalability and strong scalability of the proposed method. 
Based on the observations in the above subsection, we use the left-RAS 
preconditioner with the overlapping factor $\delta=1$ and employ the ILU(0) factorization 
with the reuse strategy as the subdomain solver in all simulations.
To accurately analyze the parallel performance,  we only run  the first $10$ time steps
with a fixed 
time step size $\Delta t =10^{-4}$. 
We first test the weak scalability of the proposed method, in which the subdomain with 
a fixed mesh size is handled by one processor core. 
In the two dimensional  test, the number of processor cores is gradually changed  from $8$ to $512$
with each processor core corresponding to a subdomain of $64\times 32 $ physical grid points.
In the three dimensional test, we set the number of processor cores in a range of $[16,\, 8,192]$, as the physical grid 
points are correspondingly changes   
from $64^3$ to $512^3$. 
 The numbers of Newton and GMRES iterations together 
with the total compute time are provided in Table \ref{tab:ws2d} for the two and three dimensional tests, respectively. The numerical results reported in Table \ref{tab:ws2d} show
 that  the averaged number of  GMRES iterations and the total 
compute time increase slowly as more processors are used for the two test cases. 
The good weak scalability of our method is validated by the simulation.

\begin{table}[!htb]
\caption{\upshape
The weak scalability of the proposed method.
}
\label{tab:ws2d} 
\centering
\begin{tabular}{|c|c|cccc|}
\hline
\multirow{5}{*}{2D test} & Mesh size   & $128^2$ & $256^2$ & $512^2$ & $1,024^2$ \\
                                      &Number of processors  & 8 & 32 & 128 & 512 \\
\cline{2-6}
  &Total Newton                  &30               &30              &30               &30  \\
  &GMRES/Newton            &6.43          &6.83            &7.07                &7.4 \\
  &Total Time (s)                 &5.46           &5.65             &5.68                &6.05 \\
                                                                    
\hline
\hline
\multirow{5}{*}{3D test} & Mesh size & $64^3$ & $128^3$ & $256^3$ & $512^3$ \\
                                     & Number of processors  &16 & 128 & 1,024 & 8,192 \\
\cline{2-6}
& Total Newton                   &20               &20              &20               &20  \\
& GMRES/Newton             &10.9            &11             &11                &11 \\
& Total Time (s)                  &140.06       &156.89       &162.76      &196.33 \\
                                                                    
\hline
\end{tabular}
\end{table}

To study the strong scalability, we run the two  and three dimensional tests 
on a $2,048^2$  mesh and a $512^3$ mesh by increasing the number of processor cores, respectively. 
As reported in Table~\ref{tab:scalability3d}, the number of nonlinear iterations is unchanged  
and the average number of linear iterations increases slightly as the number of used processor core increases. 
As shown in Fig.~\ref{scalability2d}, the total compute time decreases almost by half as the number of processor cores doubles, clearly demontrating
 a good strong parallel efficiency of the proposed algorithm. 
\begin{figure}[!th]
\begin{center}
{\includegraphics[width=0.48\textwidth]{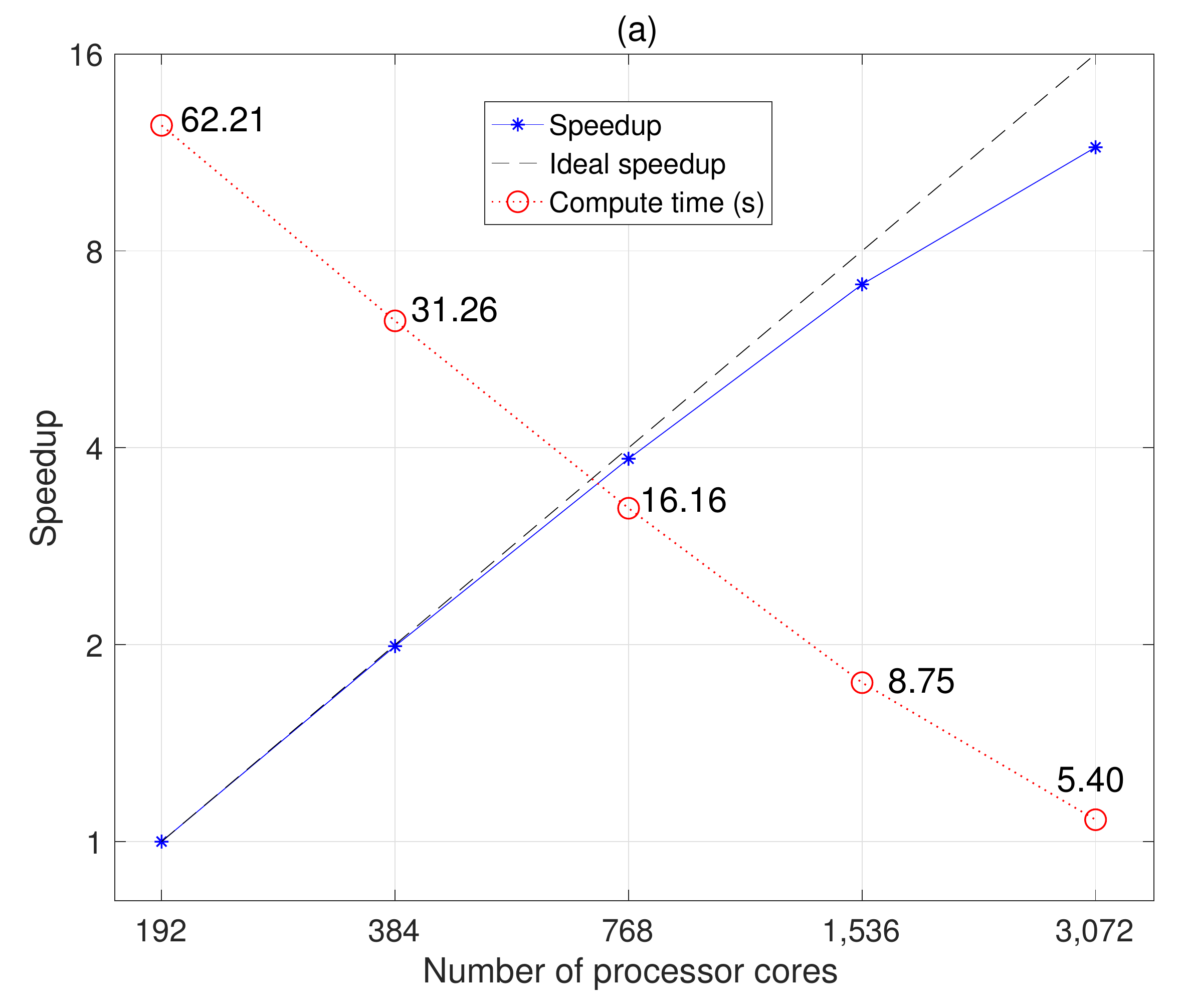}}
{\includegraphics[width=0.48\textwidth]{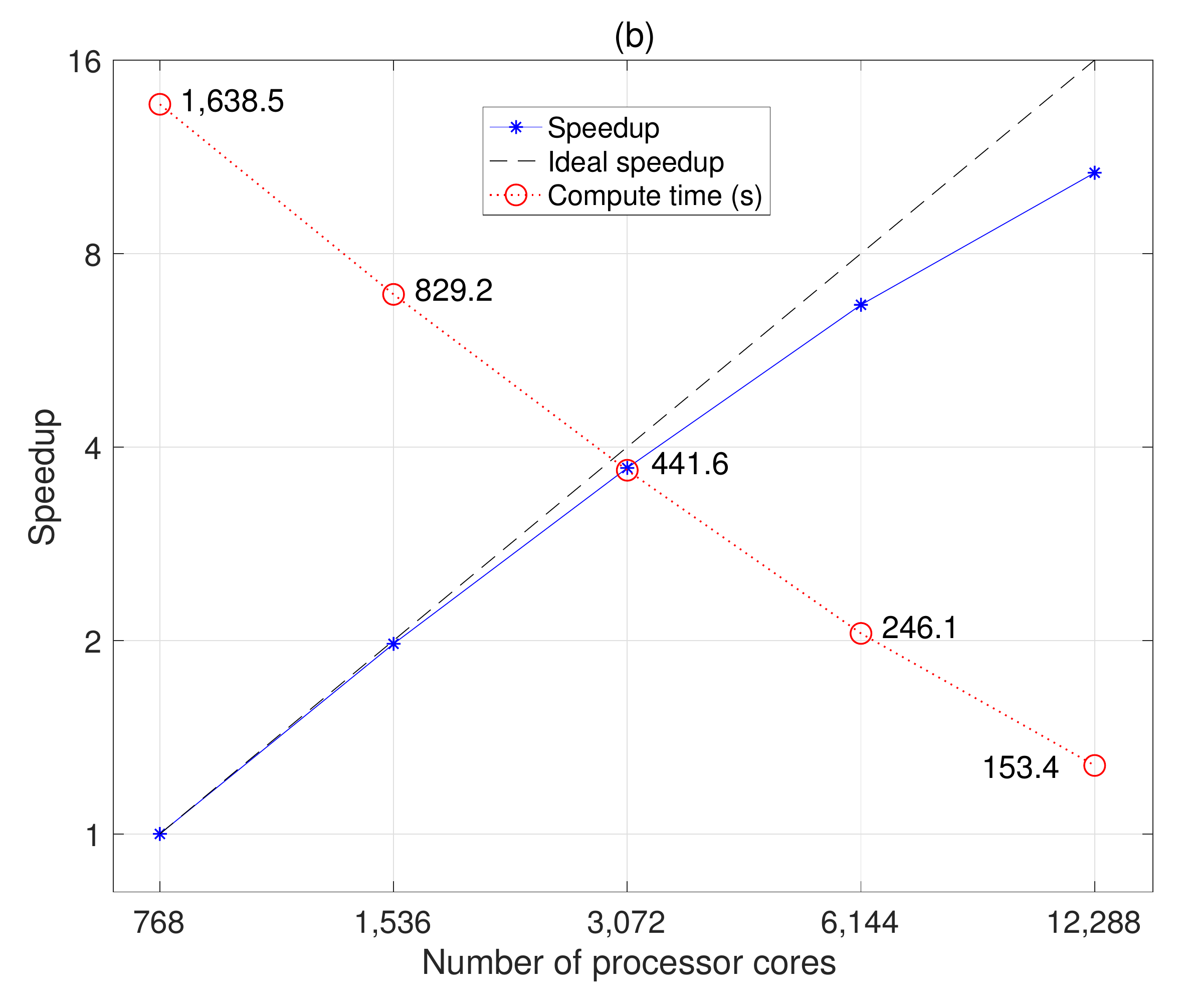}}
\end{center}
\caption{
The total compute time and the strong scalability: (a) the 2D test case, (b) the 3D test case. }
\label{scalability2d}
\end{figure}

\begin{table}[!htb]
\caption{
Performance of the NKS solver with different numbers of processor cores. Here ``NP" denotes the number of processor cores.
}
\label{tab:scalability3d} 
\centering
\begin{tabular}{|c|c|ccccc|}
\hline
 \multirow{3}{*}{2D test}&NP   & 192 & 384 & 768 &1,536 &3,072   \\
\cline{2-7}
&Total Newton        &30       &30        &30        &30           &30  \\
&GMRES/Newton  &7.1  &7.1    &7.27   &7.43     &7.67 \\
\hline
\hline
 \multirow{3}{*}{3D test}& NP & 768 &1,536 &3,072 &6,144 &12,288  \\
\cline{2-7}
&Total Newton        &20       &20        &20        &20           &20  \\
&GMRES/Newton  &10.55  &10.8     &11        &11        &11      \\
\hline
\noalign{\smallskip}
\end{tabular}
\end{table}

\section{Conclusion}
In this paper, an energy stable finite difference scheme is proposed for the coupled
Allen--Cahn/Cahn--Hilliard system.
To deal with the  logarithmic function in the total free energy,
a Taylor expansion approximation is applied to improve the numerical stability and accuracy.
We then prove that the proposed scheme is unconditionally stable and obeys the energy dissipative law.  For long time simulations, an adaptive time
stepping strategy with an automatically adjusted parameter is successfully incorporated into the energy stable
scheme such that the time step size is controlled based on the state of
solution. The nonlinear system constructed by the discretization of the AC/CH system
at each time step is solved by the NKS method. The accuracy and utility of the proposed
method is validated by several two and three dimensional test cases. Large-scale
numerical experiments show that the proposed algorithm enjoys good weak and strong scalability up to
 ten thousands processor cores on
the Sunway TaihuLight supercomputer.
We remark that  the proposed scheme for Allen--Cahn/Cahn--Hilliard system can be generalized to a much broader range of phase field equations
\color{black} with complex computational domain \color{black}
and will conduct further studies in a forthcoming paper.
\color{black}
It should also be noticed that, 
due to the possible absence of \eqref{summation}, \color{black} it is less straightforward to prove the unconditional energy stability \color{black} of the scheme 
for domains with curved geometries.   
\color{black}

\section*{Acknowledgment}

This work was supported in part by NSFC  11871069, Beijing Natural Science Foundation JQ18001, the Strategic Priority Research Program of the Chinese Academy of Sciences XDB22020100, and Beijing Academy of Artificial Intelligence.

\end{document}